\documentclass[preprint,12pt,3p]{elsarticle}
\usepackage{amsmath,amsthm,amsfonts,amssymb}
\usepackage{fullpage}
\usepackage{subfig}
\usepackage{blkarray}
\usepackage{caption}
\usepackage{float}
\usepackage{algpseudocode}
\usepackage{tikz}
\usepackage{hyperref}
\numberwithin{equation}{section}
\usepackage{xcolor}
\usepackage{colortbl}
\usepackage[normalem]{ulem}
\usepackage{sectsty}
\definecolor{astral}{RGB}{46,116,181}
\subsectionfont{\color{astral}}
\sectionfont{\color{astral}}
\usepackage{colortbl}

\DeclareMathAlphabet{\mathpzc}{OT1}{pzc}{m}{it}
\DeclareFontFamily{OT1}{pzc}{}
\DeclareFontShape{OT1}{pzc}{m}{it}{<-> s * [0.900] pzcmi7t}{}
\DeclareMathAlphabet{\mathpzc}{OT1}{pzc}{m}{it}

\usepackage{amsbsy}
\usepackage{amsmath}
\usepackage{accents}
\newlength{\dhatheight}

\newenvironment{frcseries}{\fontfamily{frc}\selectfont}{}

\newcommand{\bs}[1]{\boldsymbol{#1}}
\newcommand{\WW}{\bs{\mathcal{W}}}
\newcommand{\TT}{\bs{\mathcal{T}}}

\DeclareMathAlphabet\mathbfcal{OMS}{cmsy}{b}{n}
\definecolor{darkslategray}{rgb}{0.18, 0.31, 0.31}
\definecolor{warmblack}{rgb}{0.0, 0.26, 0.26}

\usepackage{multirow}
\usepackage{mathtools}
\usepackage{algorithm}
\usepackage{algorithmicx}
\makeatletter
\def\BState{\State\hskip-\ALG@thistlm}
\makeatother
\newtheorem{theorem}{Theorem}[section]

\newtheorem{corollary}[theorem]{Corollary}
\theoremstyle{definition}
\usepackage{hyperref}

\newtheorem{remark}{Remark}[section]
\newtheorem{example}{Example}[section]
\journal{CALCOLO}
\newcommand{\R}{{\mathbb R}}

\newcommand{\kronecker}{\raisebox{1pt}{\ensuremath{\:\otimes\:}}}
\begin{document}

\begin{frontmatter}

\title{ \textcolor{warmblack}{\bf Alternating stationary iterative methods based on double splittings}}

\cortext[cor1]{Corresponding author}

\author[label1]{Ashish Kumar Nandi}
\address[label1]{Department of Mathematics, BITS Pilani K.K. Birla Goa Campus, Goa, India}
\ead{ashish.nandi123@gmail.com}

\author[label2]{Nachiketa Mishra\corref{cor1}}
\address[label2]{Department of Mathematics, Indian Institute of Information Technology Design and Manufacturing Kancheepuram, Chennai-600127, India}
\ead{nmishra@iiitdm.ac.in}

\author[label3]{Debasisha Mishra\corref{cor1}}
\address[label3]{Department of Mathematics, National Institute of Technology Raipur, Raipur- $492010$, India}
\ead{dmishra@nitrr.ac.in}

\begin{abstract}
 Matrix double splitting iterations are simple in implementation while solving real non-singular (rectangular) linear systems.  
In this paper, we present two Alternating Double Splitting (ADS) schemes formulated by two double splittings and then alternating the respective iterations. 
The convergence conditions are then discussed along with comparative analysis. The set of double splittings used in each ADS schemes induce a preconditioned system which helps in showing the convergence of the ADS schemes. We also show that the classes of matrices for which one ADS scheme is better than the other, are mutually exclusive. Numerical experiments confirm the proposed ADS schemes are superior to the existing methods in actual implementation. Though the problems are considered in the rectangular matrix settings, the same problems are even new in non-singular matrix settings.

\end{abstract}


\end{frontmatter}
\newpage
\section{Introduction}
Most of the problems in scientific computations, solving a linear system is inevitable. Given a real
matrix $A\in{\mathbb{R}^{m\times{n}}}$ and a real vector ${b\in{\mathbb{R}^{m}}}$,  we consider the
following linear system
\begin{equation}\label{eqn1}
Ax = b, 
\end{equation}
to find an approximate solution ${x\in{\mathbb{R}^{n}}}$. In practice, these systems are large and sparse. So, the iterative methods are more suitable than direct methods. The classical iterative methods are computationally expensive, which attracts the researcher to develop fast iterative solvers. In this context, we formulate two iterative schemes using the notion of proper splittings.
 A splitting  $A = U-V$ of $A\in{\mathbb{R}^{m\times{n}}}$ is called a \textit{ proper splitting}  \cite{berp} if $R(U) = R(A)$ and $N(U) = N(A)$, where $R(U)$ and $N(U)$ denote the range space and the null space of the matrix $U$, respectively.  Different methods of construction of proper splittings are shown in Theorem 1, \cite{bern} and Theorem 3.3, \cite{ms}. In 2018, Mishra and Mishra \cite{mismis} proved the uniqueness of a proper splitting under some sufficient conditions. In 1974, Berman and Plemmons \cite{berp} considered  the following classical iterative scheme 
\begin{equation}\label{eqn1.2}
    x^{k+1} = Hx^{k}+c,
\end{equation}
as an application of proper splittings where $H=U^{\dag}V$ and $c=U^{\dag}b$.  Here $A^{\dag}$ denotes the Moore-Penrose inverse of $A$, and is defined in the next section. 
It is well-known that an iteration scheme of the form \eqref{eqn1.2} is  \textit{convergent} if the spectral radius of the iteration matrix $H$ is less than 1. Corollary 1,  \cite{berp} assures the convergence of \eqref{eqn1.2}  to $A^{\dag}b$ (the least-squares solution of minimum norm)  for any initial vector $x^{0}$. 
Several sufficient/equivalent conditions for the convergence of \eqref{eqn1.2} are reported in \cite{balim} \cite{cli1}, \cite{cli2}, \cite{jmp} and \cite{dm} for different sub-classes of proper splittings. In 2014, Jena \textit{et al.} \cite{jmp} introduced two sub-classes of proper splittings known as proper regular
splittings and proper weak regular
splittings. A proper spitting $A=U-V$ is called as a \textit{proper regular splitting} if $ U^{\dag}\geq 0$ and $V\geq 0$ (entry-wise comparison). A proper spitting $A=U-V$ is called as a \textit{proper weak regular splitting} if $ U^{\dag}\geq 0$ and $U^{\dag}V\geq 0$. Again in \cite{jmp}, the authors showed that the iterations scheme \eqref{eqn1.2} converges for a proper weak regular splitting $A=U-V$ if $A^\dag \geq 0$.  But, if a matrix has two splittings, then a splitting that yields the smaller spectral radius of the iteration matrix is preferred. In this direction, several comparison results are proved in the literature (see \cite{jmp},  \cite{mishalt1}, \cite{misarx} and \cite{mismis}). However, if a matrix has many splittings, then comparison process is time consuming.
To avoid this, Mishra \cite{misarx} in 2018 introduced the alternating iteration scheme using two proper splittings $A = U-V = M-N$, and is recalled below:
\begin{equation}\label{alt}
 x^{k+1} = U^{\dag}VM^{\dag}Nx^{k}+U^{\dag}(VM^{\dag}+I)b,
 \end{equation}
motivated by the work of \cite{benz}. Convergence  theory of \eqref{alt} can be found in \cite{misarx, mishalt1, mishalt2}.
 The idea of introducing alternating iteration scheme is inspired from the Alternating Direction Implicit (ADI) method proposed by Peaceman and Rachford \cite{PR-ADI:1955} in 1955 to solve higher dimensional Partial Differential Equations(PDEs). The notion of developing different computationally efficient methods like operator splitting method, parallel implementation of algorithms and alternating iteration schemes for linear systems are inspired from the ADI method. 
In 1959, Birkhoff and Verga \cite{BIR-VER:1959} first reformulated the ADI scheme as an iteration scheme for solving linear systems derived from the discretization of PDEs, using matrix splittings. 
 Later, the alternating scheme based algorithm is applied to a wide variety of problems, like variational problems \cite{ bruch-sloss:1985}, optimization problems and statistical learning algorithms \cite{boyd:2010, shi:2014}, alternating two-stage methods for consistent linear systems to obtain the parallel solution of Markov chains \cite{miga}, saddle-point problems and also for other different type of matrices using Hermitian and Skew-Hermitian Splitting (HSS) \cite{bai:computing, benzi:2009, damm:NLAA2000}. Further, the alternating scheme for the block matrices has been proposed in  \cite{wang:NLAA18}, by using the notion of HSS method. Our aim is to establish the convergence  theory for the alternative schemes applied to the block matrices (as shown in \eqref{eq3}) with some specific structure and properties such that the convergence is faster than the classical iteration schemes for solving the rectangular system (\ref{eqn1}).\\

At one hand, different authors in the literature focused on the problem of improving the convergence rate of the iteration scheme \eqref{eqn1.2}. On the other hand, expanding the convergence theory of the iteration scheme \eqref{eqn1.2} for different types of matrix splittings of $A$ is another topic of research interest. In this direction,
 the notion of double splitting $A=P-R+S$ of a real non-singular matrix $A$ was first introduced by Wo{\'z}nicki \cite{woz} in 1993. Such type of splitting leads to the iterative  scheme  
$$x^{k+1} = P^{-1}Rx^{k}-P^{-1}Sx^{k-1}+P^{-1}b, \;\;\; k>0$$
for solving the non-singular linear system (\ref{eqn1}), when $n = m$. 
Shen and Huang \cite{shn} and Miao {\it et al.} \cite{mio} studied the convergence and comparison of the above iterative scheme for monotone  matrices ({\it $A\in{\mathbb{R}^{n\times{n}}}$ is monotone \cite{coltz} if and only if $A^{-1}$ exists  and $A^{-1}\geq 0$}).
Moreover,  several convergence and its comparison results exist in the literature for different types of double splittings (see \cite{lcw}, \cite{lish}, \cite{lws}, \cite{msx}, \cite{shn}, \cite{shs},  \cite{sjs},  \cite{wanz}, \cite{zc}).  In 2019,  Li {\it et al.} \cite{li} proposed an alternating scheme using double splittings of a matrix to find an approximate solution of a  real non-singular linear system of equations. \\

The present article aims to revisit the theory alternating schemes using double splittings and to extend this idea to a rectangular matrix setting. In particular, we are interested in introducing another alternating scheme which we call as ADS stationary iteration scheme using double splittings like Li {\it et al.} \cite{li} and then we show that our scheme performs better in certain cases where the scheme proposed in \cite{li} fails. 
To this end, this article is organized in the following manner: Section
2 begins with the description of some useful definitions and preliminary results.
Section
3 proposes two ADS schemes and analyzes its convergence criteria.  Section
4 shows the performance of the proposed iteration scheme
by extensive numerical examples. 

\section{Prerequisites}
In this section, additional notations, definitions and useful results related to non-negative matrices and double proper splittings are presented which are virtually used throughout this article.  We denote $\mathbb{R}^{m\times{n}}$  the set of all real rectangular matrices of order $m\times n$ and  $\mathbb{R}^{n}$ is an $n$-dimensional Euclidean space. The {\it rank} of a matrix $A \in \mathbb{R}^{m\times{n}}$ is denoted by $r(A)$.  Suppose $L$ and $M$ are two complementary subspaces of $\mathbb{R}^{n}$. 
Let $\tilde{P}_{L,M}$ be the projection on $L$ along $M$. Hence $\tilde{P}_{L,M}A = A$ if and only if $R(A)\subseteq L$ and $A\tilde{P}_{L,M} = A$ if and only if $N(A)\supseteq M$.  
For $A\in\mathbb{R}^{m\times{n}}$, the unique matrix $X\in\mathbb{R}^{n\times{m}}$ satisfying the conditions $AXA=A,~XAX=X,~(AX)^{t}=AX~\text{and}~ (XA)^{t}=XA$ is called the {\it Moore-Penrose inverse} of $A$, where $A^{t}$ denotes the transpose of the matrix $A$. The Moore-Penrose inverse always exists, and is denoted by $A^{\dag}$. The matrix $A\in\mathbb{R}^{m\times{n}}$ is called {\it semi-monotone} if $A^{\dag} \geq 0.$ The properties of $A^{\dag}$ which are frequently used in this article: $R(A^{\dag}) = R(A^{t})$; $N(A^{\dag}) = N(A^{t})$; $AA^{\dag} = \tilde{P}_{R(A)}$; $A^{\dag}A = \tilde{P}_{R(A^{t})}$.

\subsection{Spectral radius and  non-negative matrices}\label{sub2.01}
We denote  the set of all eigenvalues of $A\in {\R}^{n \times n}$ as $\sigma(A)$.
The {\it spectral radius} of $A\in\mathbb{R}^{n\times{n}}$, denoted by  $\rho (A)$, is defined as $\rho (A) = \displaystyle{ \max\limits_{1\leq j\leq n} |\lambda_{j}|}$, where $\lambda_{j} \in \sigma(A).$ $A \in {\R}^{m \times n}$ is
called \textit{non-negative} if
 $A \geq 0$.    Let $B,C \in {\R}^{m \times n}$. We write $B \geq C$ if $B-C \geq 0$.
    The next results deal with non-negativity of a matrix and the spectral radius. 
    
\begin{theorem}[Theorem 2.1.11, \cite{bpn}]\label{2.1.2}
Let  $B\in{\mathbb{R}^{n\times{n}}}$, $B\geq 0$, $x\geq 0$ $(x\neq 0)$ and $\alpha$ is a positive scalar. If $\alpha x\leq Bx$, then $\alpha\leq \rho(B)$.
\end{theorem}

\begin{theorem}[Lemma 2.2, \cite{shn}]\label{2.1.3}
Let ${\bf A }= \begin{pmatrix}
    B & C\\ 
    I & 0
    \end{pmatrix}\geq 0$ and $\rho(B+C)<1$. Then, $\rho({\bf A })<1.$
\end{theorem}

\begin{theorem}[Theorem 2.20, \cite{var}]\label{2.1.4}
Let $A\in{\mathbb{R}^{n\times{n}}}$ and $A \geq 0$. Then\\
$(i)$ $A$ has a non-negative real eigenvalue equal to its spectral radius.\\
$(ii)$ there exists a non-negative eigenvector for its spectral radius.
\end{theorem}
\newpage
\subsection{Double proper splittings}\label{sub2.02}

Motivated by the standard iterative methods like Jacobi,
Gauss-Seidel, SOR etc., Wo\'{z}nicki   \cite{woz} introduced double splitting theory for finding iteration solution of non-singular linear system $Ax=b$. Neumann \cite{neum} extended the non-singular case to singular linear system which he named as 3-part splitting. A double splitting $A = P-R+S$ of $A\in{\mathbb{R}^{m\times{n}}}$ is called {\it double proper splitting} if $R(P) = R(A)$ and $N(P) = N(A)$.
Further, Jena \textit{et al.} \cite{jmp} introduced two subclasses of double proper splittings which are recalled below.
A double proper splitting $A = P-R+S$ is called a \textit{double proper regular splitting} \cite{jmp} if $P^{\dag}\geq 0$, $R\geq 0$ and $S\leq 0$.
 The next subclass contains the above one.
 A double proper splitting $A = P-R+S$ is called a \textit{double proper weak regular splitting } \cite{jmp} if $P^{\dag}\geq 0$, $P^{\dag}R\geq 0$ and $P^{\dag}S\leq 0$.
Mishra \cite{dm} again introduced another subclass which contains the above two subclasses. He named it as double proper nonnegative splitting. However, we call the same as double proper weak splitting as the conditions are weaker than the earlier two.
 A double proper splitting $A = P-R+S$ is called \textit{double proper weak splitting} if $P^{\dag}R\geq 0$ and $P^{\dag}S\leq 0$.
  In the non-singular matrix setting, the above definitions coincide with double regular splitting (or regular double splitting \cite{shn}), double weak regular  splitting (or weak regular double splitting \cite{shn}),  and double weak splitting (or double nonnegative splitting \cite{sjs}), respectively. 
 Analogous to the non-singular case,  the following iterative scheme spanned in three iterates (known as \textit{double iteration scheme}) is proposed by Jena \textit{et al.} \cite{jmp} by the help of double proper splitting $A = P-R+S$:
\begin{equation}\label{eq2}
 x^{k+1} = P^{\dag}Rx^{k}-P^{\dag}Sx^{k-1}+P^{\dag}b, ~~k>0.
 \end{equation}
The equivalent block-matrix form \cite{jmp} of (\ref{eq2}) is\\
\begin{equation}\label{eq3}
{\bf x^{k+1}} = {\bf T } {\bf x^k } + {\bf b},
\end{equation}
where ${\bf x^{k+1}}= \begin{pmatrix}
    x^{k+1} \\
    x^{k}
  \end{pmatrix}$, ${\bf x^{k}}= \begin{pmatrix}
    x^{k} \\
    x^{k - 1}
  \end{pmatrix}$, ${\bf T}= \begin{pmatrix}
    P^{\dag}R & -P^{\dag}S \\
    I & 0
  \end{pmatrix}$, ${\bf b}= \begin{pmatrix}
    P^{\dag}b \\
    0
  \end{pmatrix}$
and $I$ denotes the identity matrix of order $n$. Then, the iteration scheme (\ref{eq3})  converges to  $A^{\dag}b$ of (\ref{eqn1}) if $\rho({\bf T })<1$.
Here the spectral radius of block matrix {\bf T} is the spectral radius of the full matrix $T$. Rest of the manuscript, we will write $\rho({ T })$ instead of $\rho({\bf T })$ for any block matrix {\bf T.}
The next two results present the convergence criteria for double proper regular (or weak regular) splittings and double proper weak  splittings. But, interested reader may refer    \cite{jmp}, \cite{dm}, \cite{balim} and \cite{kur} for more detailed convergence theory of \eqref{eq3}.

\begin{theorem}[Theorem 3.6, \cite{jmp}]\label{2.1.1}
Let $A^{\dag}\geq 0$. If $A = P-R+S$ be a double proper regular (or weak regular) splitting  of $A\in{\mathbb{R}^{m\times{n}}}$, then $\rho(T)<1.$
\end{theorem}

\begin{theorem}[Theorem 4.5, \cite{dm}]\label{2.1.10}
Let $A^{\dag}P\geq 0$. If $A = P-R+S$ be a double proper weak splitting of $A\in{\mathbb{R}^{m\times{n}}}$, then $\rho( T )<1.$
\end{theorem}

\newpage
\section{Main results}

\subsection{Formulation of Alternating Double Splitting (ADS) schemes}\label{sub3.01}

Motivated by the work of Li {\it et al.} \cite{li} where the authors introduced Alternating Double Splitting (ADS) scheme using double splittings to solve a non-singular linear system, and the work of Jena \textit{et al.}  \cite{jmp},  we consider two double iterative schemes with respect to two double proper splittings of $A\in{\mathbb{R}^{m\times{n}}}$, respectively  as: $A = P_1-R_1+S_1 = P_2-R_2+S_2$ are 
\begin{equation}\label{eq4.1}
  x^{k+1/2} = P_1^{\dag}R_1x^{k}-P_1^{\dag}S_1x^{k-1/2}+P_1^{\dag}b,  
\end{equation}and 
\begin{equation}\label{eq4.2}
  x^{k+1} = P_2^{\dag}R_2x^{k+1/2}-P_2^{\dag}S_2x^{k}+P_2^{\dag}b.
\end{equation}
The corresponding block iterative schemes can be written in two different ways as mentioned below for $i=1,2$:
$${\bf T}_i= \begin{pmatrix}
    {P_i}^{\dag}R_i & -{P_i}^{\dag}S_i \\
    I & 0
  \end{pmatrix}, {\bf G}_i= \begin{pmatrix}
    I & 0 \\
    {P_i}^{\dag}{R_i} & -{P_i}^{\dag}S_i
  \end{pmatrix} \mbox{ and } {\bf H}_i = \begin{pmatrix}
    {P_i}^{\dag}{R_i} & -{P_i}^{\dag}S_i \\
    0 & I
  \end{pmatrix}.$$
  From each pair of block forms, we are going to formulate next a ADS scheme. 
\subsubsection{TG-ADS scheme}
\begin{equation}\label{eqq1}
\begin{cases} 
{\bf x}^{k+1/2} &= 
\begin{pmatrix}
    x^{k+1/2} \\
    x^{k}
  \end{pmatrix}
  = \begin{pmatrix}
    I & 0\\
    P_1^{\dag}R_1 & -P_1^{\dag}S_1 
    \end{pmatrix} \begin{pmatrix}
    x^{k} \\
    x^{k-1/2}
  \end{pmatrix}
  +  \begin{pmatrix}
     0\\
     P_1^{\dag}b
  \end{pmatrix}\\
  &= {\bf G}_1{\bf x}^{k}+{\bf b}_1 \\[2ex]
{\bf x}^{k+1} &= 
\begin{pmatrix}
    x^{k+1} \\
    x^{k+1/2}
  \end{pmatrix}
  = \begin{pmatrix}
    P_2^{\dag}R_2 & -P_2^{\dag}S_2\\ 
    I & 0
    \end{pmatrix} \begin{pmatrix}
    x^{k+1/2} \\
    x^{k}
  \end{pmatrix}
  +  \begin{pmatrix}
     P_2^{\dag}b\\
     0
     \end{pmatrix}\\
  &= {\bf T}_2{\bf x}^{k + 1/2}+{\bf b}_2.
  \end{cases}
\end{equation}
To do the convergence analysis of   (\ref{eqq1}), we next formulate a single-step double iteration  scheme by composing the half-step double iteration schemes in (\ref{eqq1}).
\begin{equation}\label{eq5}
 {\bf x}^{k+1} = 
 {\bf T}_2{\bf G}_{1}{\bf x}^{k}+{\bf T}_2 {\bf b}_1 + {\bf b}_2={\bf W}_{12}{\bf x}^{k}+{\bf b}_3.
\end{equation}
We call the above scheme as \textit{TG-ADS scheme}. The iteration matrix ${\bf W}_{12}$  and the vector ${\bf b_3}$ of the TG-ADS scheme are as follows: 
\begin{equation*}
 {\bf W}_{12} = \begin{pmatrix}
    P_2^{\dag}R_2-P_2^{\dag}S_2P_1^{\dag}R_1 & P_2^{\dag}S_2P_1^{\dag}S_1\\ 
    I & 0
    \end{pmatrix} \text{ and } {\bf b_3} = \begin{pmatrix}
     P_2^{\dag}(I - S_2P_1^{\dag})b\\
     0
  \end{pmatrix}.
\end{equation*}
\subsubsection{HT-ADS scheme}
   
Alike the half-step double iteration schemes used in TG-ADS scheme, we introduce a new ADS scheme by defining another pair of half-step double iteration schemes.

\begin{equation}\label{eqq2}
\begin{cases} 
{\bf x}^{k+1/2} &=
\begin{pmatrix}
    x^{k+1/2} \\
    x^{k}
  \end{pmatrix}
  = \begin{pmatrix}
    P_1^{\dag}R_1 & -P_1^{\dag}S_1 \\
    I & 0\\
    \end{pmatrix} \begin{pmatrix}
    x^{k} \\
    x^{k-1/2}
  \end{pmatrix}
  +  \begin{pmatrix}
     P_1^{\dag}b \\
     0
  \end{pmatrix}\\[2ex]
  &= {\bf T}_1{\bf x}^{k}+{\bf b}_1 \\[2ex]

{\bf x}^{k+1} &= 
\begin{pmatrix}
    x^{k+1} \\
    x^{k+1/2}
  \end{pmatrix}
  = \begin{pmatrix}
    P_2^{\dag}R_2 & -P_2^{\dag}S_2\\ 
    0 & I
    \end{pmatrix} \begin{pmatrix}
    x^{k+1/2} \\
    x^{k}
  \end{pmatrix}
  +  \begin{pmatrix}
     P_2^{\dag}b\\
     0
     \end{pmatrix}\\[2ex]
  &= {\bf H}_2{\bf x}^{k+\frac{1}{2}}+{\bf b}_2.
  \end{cases}
\end{equation}
 The corresponding single-step double iteration scheme is derived as follows:
\begin{equation}\label{pp1.1}
{\bf x}^{k+1} = {\bf H}_2{\bf T}_{1}{\bf x}^{k}+{\bf H}_2 {\bf b}_1 + {\bf b}_2 = {{\WW}_{12}}{\bf x}^{k}+{\bf b_4}.
\end{equation}
This scheme is called as  \textit{HT-ADS scheme}. The iteration matrix ${\bf \WW}_{12}$  and the vector ${\bf b_4}$ of the HT-ADS scheme are 
\begin{equation*}
     {\WW}_{12} = \begin{pmatrix}
    P_2^{\dag}R_2P_1^{\dag}R_1-P_2^{\dag}S_2 & -P_2^{\dag}R_2P_1^{\dag}S_1\\ 
    I & 0
    \end{pmatrix} \text{ and } {\bf b_4} = \begin{pmatrix}
     P_2^{\dag}(R_2P_1^{\dag}+I)b\\
     0
  \end{pmatrix},
\end{equation*}
respectively. The iteration schemes $(\ref{eq5})$ and $(\ref{pp1.1})$ are called as {\it ADS alternating iteration schemes} (ADS schemes)  in its block form. 
\begin{remark}
HT-ADS scheme in its block form  yields a three-term recurrence scheme  
\begin{equation}\label{DAHS2}
 x^{k+1} = (P_2^{\dag}R_2P_1^{\dag}R_1-P_2^{\dag}S_2)x^{k}-P_2^{\dag}R_2P_1^{\dag}S_1x^{k-1}+P_2^{\dag}R_2P_1^{\dag}b+P_2^{\dag}b,     
 \end{equation} 
 which is also formed by eliminating $x^{k+1/2}$ from (\ref{eq4.2}). However, one can verify that TG-ADS scheme in its block form which extends the Alternating Double Splitting method proposed by Li {\it et al.} \cite{li} does not coincide with \eqref{DAHS2}.
\end{remark}

The iteration schemes \eqref{eq5} and \eqref{pp1.1} converge for any initial guess ${\bf x}^{0}$ to  $A^{\dag}b$ if and only if $\rho(W_{12})<1$ and $\rho(\mathcal{W}_{12})<1$, respectively \cite{jmp}. The next section provides different sufficient conditions for the convergence of the above types of ADS schemes.

\subsection{Convergence analysis}\label{sec:cgs}

We show the convergence of each ADS scheme
 by considering the spectral radius of another iteration matrix for solving a new preconditioned system as both the iteration matrices have the same spectral radius. This is shown next.
\subsubsection{TG-ADS scheme}\label{TG:cgs}
Let $A = P_1-R_1+S_1 = P_2-R_2+S_2$ be two double proper splittings of $A\in{\mathbb{R}^{m\times{n}}}$ with $N(S_2)\supseteq N(P_2)$, $R(S_2)\subseteq R(P_2)$ and $1 \notin \sigma(S_2P_1^{\dag})$. So, $I-S_2P_1^{\dag}$ is non-singular. Let us consider the preconditioned linear system 
\begin{equation}\label{pc1}
    \widehat{A}x = \widehat{b},
\end{equation}
where $\widehat{A} = (I-S_2P_1^{\dag})A$ and $\widehat{b} = (I-S_2P_1^{\dag})b$.  Simplifying $\widehat{A}$, we have
\begin{eqnarray*}
\widehat{A} &=&  (I-S_2P_1^{\dag})A\\
&=& A-S_2P_1^{\dag}A\\
&=& P_2-R_2+S_2-S_2P_1^{\dag}(P_1-R_1+S_1)\\
&=& P_2-R_2+S_2-S_2+S_2P_1^{\dag}R_1-S_2P_1^{\dag}S_1\\
&=& P_2-(R_2-S_2P_1^{\dag}R_1)+(-S_2P_1^{\dag}S_1)\\
&=& \widehat{P}-\widehat{R}+\widehat{S}
\end{eqnarray*}
is a double splitting of $\widehat{A}$. For convenience, we denote $\widehat{P} = P_2$, $\widehat{R} = R_2-S_2P_1^{\dag}R_1$ and $\widehat{S} = -S_2P_1^{\dag}S_1$. Next, we have to show that $\widehat{A} = \widehat{P}-\widehat{R}+\widehat{S}$ is a double proper splitting of $\widehat{A}$.  Let $x\in N(\widehat{A})$. This implies $\widehat{A}x = 0$, i.e., $(I-S_2P_1^{\dag})Ax = 0$. Pre-multiplying $(I-S_2P_1^{\dag})^{-1}$ to $(I-S_2P_1^{\dag})Ax = 0$ yields $Ax = 0$. So $N(\widehat{A}) \subseteq N(\widehat{P})$. Again, suppose that $x\in  N(\widehat{P})= N(P_2) = N(A)$. This gives $Ax =0$ which yields $(I-S_2P_1^{\dag})^{-1}\widehat{A}x = 0$. So, we get
$\widehat{A}x = 0$ which implies $N(\widehat{P}) \subseteq N(\widehat{A})$. Hence $N(\widehat{A}) = N(\widehat{P})$.
Next, to show that $R(\widehat{A}) = R(\widehat{P})$. From (\ref{pc1}), we obtain $\widehat{A} =  A-S_2P_1^{\dag}A = A-P_2P_2^{\dag}S_2P_1^{\dag}A = A-AA^{\dag}S_2P_1^{\dag}A = A(I-A^{\dag}S_2P_1^{\dag}A)$. This gives $R(\widehat{A})\subseteq R(A) = R(P_2) = R(\widehat{P})$. Also, $r(\widehat{A}) = r(\widehat{P})$. Hence $R(\widehat{A}) = R(\widehat{P})$.  Thus, $\widehat{A} = \widehat{P}-\widehat{R}+\widehat{S}$ is a  double proper splitting of $\widehat{A}$ and the corresponding double iterative scheme for the preconditioned system (\ref{pc1}) can be written as
\begin{equation}\label{preeq1}
  x^{k+1} = \widehat{P}^{\dag}\widehat{R}x^{k}-\widehat{P}^{\dag}\widehat{S}x^{k-1}+\widehat{P}^{\dag}\widehat{b},~~k>0   
\end{equation}
i.e.,
\begin{equation*}
{\bf x}^{k+1}
  = \begin{pmatrix}
    \widehat{P}^{\dag}\widehat{R} & -\widehat{P}^{\dag}\widehat{S} \\
    I & 0
  \end{pmatrix} 
  {\bf x}^{k}
  +  \begin{pmatrix}
    \widehat{P}^{\dag}\widehat{b} \\
    0
  \end{pmatrix}.
\end{equation*}
The iteration matrix is
$$\widehat{\bf T} = \begin{pmatrix}
    \widehat{P}^{\dag}\widehat{R} & -\widehat{P}^{\dag}\widehat{S} \\
    I & 0
  \end{pmatrix}\\
    =
    \begin{pmatrix}
    P_2^{\dag}R_2-P_2^{\dag}S_2P_1^{\dag}R_1 & P_2^{\dag}S_2P_1^{\dag}S_1\\ 
    I & 0
    \end{pmatrix}\\
    = {\bf W}_{12}$$
and 
\begin{eqnarray*}
 \begin{pmatrix}
    \widehat{P}^{\dag}\widehat{b} \\
    0
  \end{pmatrix}
  &=& \begin{pmatrix}
    P_2^{\dag}(I-S_2P_1^{\dag})b \\
    0
  \end{pmatrix}\\
  &=& \begin{pmatrix}
    P_2^{\dag}R_2 & -P_2^{\dag}S_2\\ 
    I & 0
    \end{pmatrix}\begin{pmatrix}
    0\\
    P_1^{\dag}b
  \end{pmatrix}+\begin{pmatrix}
    P_2^{\dag}b\\
    0
  \end{pmatrix} = {\bf T}_2{\bf b}_1+{\bf b}_2.
\end{eqnarray*}
 We remark that
 if $A= P_1-R_1+S_1 = P_2-R_2+S_2$ are two double proper regular (or weak regular or weak) splittings of $A\in{\mathbb{R}^{m\times{n}}}$ with $N(S_2)\supseteq N(P_2)$, $R(S_2)\subseteq R(P_2)$ and $1 \notin \sigma(S_2P_1^{\dag})$, then $\widehat{A} = \widehat{P}-\widehat{R}+\widehat{S}$ is also a double proper regular (weak regular or  weak) splittings of $A\in{\mathbb{R}^{m\times{n}}}$.
 Hence, an immediate consequence of Theorem \ref{2.1.1} which generalizes Theorem $2.6$, \cite{li} is as follows.

\begin{theorem}\label{conv2}
Let $A= P_1-R_1+S_1 = P_2-R_2+S_2$ be two double proper regular (weak regular) splittings of a semi-monotone matrix $A\in{\mathbb{R}^{m\times{n}}}$. If $N(S_2)\supseteq N(P_2)$, $R(S_2)\subseteq R(P_2)$, $1 \notin \sigma(S_2P_1^{\dag})$ and $\widehat{A}^{\dag}\geq 0$, then $\rho(W_{12})<1$.  
\end{theorem}

The next example shows that the converse of the above theorem is not true.

\begin{example}\label{ex1.11}
Let $A = \begin{bmatrix}
\frac{3}{27} & -\frac{5}{54} \\
-\frac{5}{54} & \frac{3}{27}
\end{bmatrix} = P_1-R_1+S_1 = P_2-R_2+S_2$ be two double regular splittings of a monotone matrix $A$, where 

\begin{equation*}
P_1 = 
\begin{bmatrix}
\frac{4}{27} & -\frac{2}{27} \\
-\frac{2}{27} & \frac{4}{27}
\end{bmatrix},
R_1 = 
\begin{bmatrix}
 \frac{1}{27} & 0 \\
  0 & \frac{1}{27}
\end{bmatrix},
S_1 = 
\begin{bmatrix}
0 & -\frac{1}{54} \\
-\frac{1}{54} & 0
\end{bmatrix},
\end{equation*}

\begin{equation*}
    P_2 = \begin{bmatrix}
 \frac{5}{27} & \frac{5}{27} \\
\frac{5}{27} & \frac{5}{27}
\end{bmatrix}, 
R_2 = \begin{bmatrix}
\frac{2}{27} & \frac{3}{54} \\
\frac{3}{54} & \frac{2}{27}
\end{bmatrix},
S_2 = \begin{bmatrix}
0 & -\frac{2}{9} \\
-\frac{2}{9} & 0
\end{bmatrix}.
\end{equation*}
Here $S_2P_1^{-1} = \begin{bmatrix}
 -1 & -2 \\
 -2 & -1
\end{bmatrix}.$
We have $\rho(W_{12}) = 0.8306<1$ but $1 \in \sigma(S_2P_1^{-1})$.
\end{example}

Next result discusses the case when $A$ has two double proper weak splittings. This extends Theorem $2.4$, \cite{li} to rectangular matrices.

\begin{theorem}\label{conv3}
Let $A= P_1-R_1+S_1 = P_2-R_2+S_2$ be two double proper weak splitting of $A\in{\mathbb{R}^{m\times{n}}}$. If $\widehat{A}^{\dag}\widehat{P}\geq 0$, $N(S_2)\supseteq N(P_2)$, $R(S_2)\subseteq R(P_2)$ and $1 \notin \sigma(S_2P_1^{\dag})$, then $\rho(W_{12} )<1$.
\end{theorem}

\subsubsection{HT-ADS scheme}\label{HT:cgs}
Let $A = P_1-R_1+S_1 = P_2-R_2+S_2$ be two double proper regular (weak regular) splittings of $A\in{\mathbb{R}^{m\times{n}}}$ such that $N(R_2)\supseteq N(P_2)$, $R(R_2)\subseteq R(P_2)$ and $-1 \notin \sigma(R_2P_1^{\dag})$. 
We then get another preconditioned linear system  
\begin{align}\label{pc2}
    \widehat{\mathcal{A}}x = \widehat{b}_1,
\end{align}
where $\widehat{\mathcal{A}} = (I+R_2P_1^{\dag})A$.
Proceeding similarly as in the convergence analysis discussed in the subsection \ref{TG:cgs}, we thus have $\widehat{\mathcal{A}} = \widehat{\mathcal{P}} - \widehat{\mathcal{ R}}+\widehat{\mathcal{ S}}$ is a  double proper regular (weak regular) splitting of $\widehat{\mathcal{A}}$ where  $\widehat{\mathcal{P}} = P_2$, $\widehat{\mathcal{R}} = R_2P_1^{\dag}R_1-S_2$ and $\widehat{\mathcal{S}} = R_2P_1^{\dag}S_1$. 
   The iteration matrix of the double iteration scheme \eqref{preeq1} with respect to the double proper splitting $\widehat{\mathcal{A}} = \widehat{\mathcal{P}} - \widehat{\mathcal{ R}}+\widehat{\mathcal{ S}}$ is
\begin{equation*}
 {\bf \widehat{\TT}} = \begin{pmatrix}
    \widehat{\mathcal{P}}^{\dag}\widehat{\mathcal{R}} & -\widehat{\mathcal{P}}^{\dag}\widehat{\mathcal{S}}\\ 
    I & 0
    \end{pmatrix}  =  \begin{pmatrix}
    P_2^{\dag}R_2P_1^{\dag}R_1-P_2^{\dag}S_2 & -P_2^{\dag}R_2P_1^{\dag}S_1\\ 
    I & 0
    \end{pmatrix}  = {\bf {\WW}_{12}}.
\end{equation*}
We therefore have the following convergence theorem for the HT-ADS scheme by using Theorem \ref{2.1.1}.

\begin{theorem}\label{wwcon}
If  $A = P_1-R_1+S_1 = P_2-R_2+S_2$ be two double proper regular (weak regular) splittings such that $N(R_2)\supseteq N(P_2)$, $R(R_2)\subseteq R(P_2)$, $-1 \notin \sigma(R_2P_1^{\dag})$ and $\widehat{\mathcal{A}}^{\dag}\geq 0$, then $\rho({\mathcal{W}}_{12})<1.$ \end{theorem}

Note that the condition $\widehat{\mathcal{A}}^{\dag}\geq 0$ will be replaced by $\widehat{\mathcal{A}}^{\dag}\widehat{\mathcal{P}}\geq 0$ in the case of $A$ having double proper weak splittings. 

\subsection{Comparison Results: TG-ADS scheme}
Convergence theory of ADS schemes will be meaningful if the proposed ADS schemes \eqref{eq5} and \eqref{pp1.1} converge faster than the two individual double iteration schemes of the form \eqref{eq2}. This is discussed first in Theorem \ref{Hybcomp} before moving into other problems. In this context, the following question arises now which is highly useful in practice, i.e., how to choose the second double splitting $A = P_2-R_2+S_2$ if  $A = P_1-R_1+S_1$ is given such that the TG-ADS scheme converges faster than the double iteration scheme arising out of  the splitting  $A = P_1-R_1+S_1$. This is addressed in the next result. 

\begin{theorem}
\label{altcomp121}
Let  $A = P_1-R_1+S_1$ be a double proper  weak regular splitting and $A = P_2-R_2+S_2$ be a double proper  regular  splitting of a  semi-monotone matrix $A\in{\mathbb{R}^{m\times{n}}}$. Suppose that $N(S_2)\supseteq N(P_2)$, $R(S_2)\subseteq R(P_2)$, $1 \notin \sigma(S_2P_1^{\dag})$ and $\widehat{A}^{\dag}\geq 0$. If $P_1^{\dag}R_1\geq P_2^{\dag}R_2$ and $P_2^{\dag}S_2\geq P_1^{\dag}S_1$, then $\rho(W_{12})\leq \rho({T}_1)<1.$ 
\end{theorem}
\begin{proof}
By Theorem \ref{2.1.1} and Theorem \ref{conv2}, we have $\rho({T}_1)<1$ and $\rho(W_{12})<1$, respectively.\\
Case (i): $\rho(W_{12}) = 0$.  The proof is obvious.\\ 
Case (ii): $\rho(W_{12}) \neq 0$. Since ${\bf W}_{12}\geq 0$,  there exists a non-negative eigenvector
${\bf x} = \begin{pmatrix}
    x_1 \\
    x_2
  \end{pmatrix} $
such that ${\bf W}_{12}{\bf x} = \rho(W_{12}){\bf x}$ by Theorem \ref{2.1.4}. This gives
\begin{align}
(P_2^{\dag}R_2-P_2^{\dag}S_2P_1^{\dag}R_1)x_1+P_2^{\dag}S_2P_1^{\dag}S_1x_2 = \rho(W_{12})x_1\label{th3.4eq1}\\
x_1  = \rho(W_{12})x_2.\label{th3.4eq2}
\end{align}
We next have
\begin{equation*}\label{th3.4eq3}
 (P_2P_2^{\dag}R_2-P_2P_2^{\dag}S_2P_1^{\dag}R_1)x_1+ \frac{1}{\rho(W_{12})}P_2P_2^{\dag}S_2P_1^{\dag}S_1x_1 = \rho(W_{12})P_2x_1   
\end{equation*}
by pre-multiplying $P_2$ in (\ref{th3.4eq1}). Also, $P_2P_2^{\dag}R_2 = R_2$ as $R(R_2)\subseteq R(P_2)$ which follows from $R(S_2)\subseteq R(P_2)$ and $R(A)= R(P_2)$, and $P_2P_2^{\dag}S_2 = S_2$ as $R(S_2)\subseteq R(P_2)$. Using (\ref{th3.4eq2}), it then 
yields
\begin{equation}\label{eq3.3}
    \rho(W_{12})^{2}P_2x_1 = \rho(W_{12})(R_2-S_2P_1^{\dag}R_1)x_1+S_2P_1^{\dag}S_1x_1,
\end{equation}
which results $P_2x_1\geq 0$.
Also, we have
\begin{align*}
0 &= \rho(W_{12})^{2}P_2x_1 -\rho(W_{12})(R_2-S_2P_1^{\dag}R_1)x_1-S_2P_1^{\dag}S_1x_1 &~&\\
&\leq \rho(W_{12})P_2x_1 -\rho(W_{12})(R_2-S_2P_1^{\dag}R_1)x_1-\rho(W_{12})S_2P_1^{\dag}S_1x_1 &~&\\
&= \rho(W_{12})\left(P_2-R_2+S_2P_1^{\dag}(R_1-S_1)\right)x_1 &~&\\
&= \rho(W_{12})\left(A-S_2+S_2P_1^{\dag}(P_1-A)\right)x_1 &~&\\
&= \rho(W_{12})\left(A-S_2+S_2P_2^{\dag}P_2-S_2P_1^{\dag}A\right)x_1 ~~\text{ by } P_1^{\dag}P_1 = P_2^{\dag}P_2  ~\mbox{as}~R(P_1)=R(A)=R(P_2)\\
&= \rho(W_{12})\left(A-S_2P_1^{\dag}A\right)x_1 ~~~~~~~~~~~~~~~~~~~~~~\, (\because S_2P_2^{\dag}P_2 = S_2 ~\mbox{as} ~N(S_2)\supseteq N(P_2))\\
&= \rho(W_{12})\left(I+(-S_2P_1^{\dag})\right)Ax_1.
\end{align*}
Thus, $Ax_1\geq 0$.
Now 
\begin{eqnarray*}
{\bf T}_1{\bf x}-\rho(W_{12}){\bf x} &=& \begin{pmatrix}
    P_1^{\dag}R_1x_1-P_1^{\dag}S_1x_2-\rho(W_{12})x_1\\ 
    x_1-\rho(W_{12})x_2
    \end{pmatrix}\\  
&=& \begin{pmatrix}
   \Delta\\ 
    0
    \end{pmatrix},
\end{eqnarray*}
where
\begin{eqnarray*}
\Delta
&=& (P_1^{\dag}R_1-P_2^{\dag}R_2)x_1+P_2^{\dag}S_2P_1^{\dag}R_1x_1-\frac{1}{\rho(W_{12})}P_1^{\dag}S_1x_1-\frac{1}{\rho(W_{12})}P_2^{\dag}S_2P_1^{\dag}S_1x_1\\
&\geq& P_2^{\dag}S_2P_1^{\dag}R_1x_1-\frac{1}{\rho(W_{12})}P_1^{\dag}S_1x_1-\frac{1}{\rho(W_{12})}P_2^{\dag}S_2P_1^{\dag}S_1x_1\\
&\geq& \frac{1}{\rho(W_{12})}\left(P_2^{\dag}S_2P_1^{\dag}(R_1-S_1)\right)x_1-\frac{1}{\rho(W_{12})}P_1^{\dag}S_1x_1\\
&=& \frac{1}{\rho(W_{12})}\left(P_2^{\dag}S_2P_1^{\dag}(P_1-A)-P_1^{\dag}S_1\right)x_1\\
&=& \frac{1}{\rho(W_{12})}(P_2^{\dag}S_2-P_1^{\dag}S_1)x_1+\frac{1}{\rho(W_{12})}(-P_2^{\dag}S_2)P_1^{\dag}Ax_1\geq 0.
\end{eqnarray*}

 Therefore, ${\bf T}_1{\bf x}-\rho(W_{12}){\bf x} \geq 0$. By Theorem \ref{2.1.2}, we thus have $\rho(W_{12})\leq \rho({T}_1)<1.$
\end{proof}

Next result shows that the ADS scheme performs  better than the other double iteration scheme formed by $A = P_2-R_2+S_2$. This extends Theorem $3.5$, \cite{li} to rectangular matrix case and  can be proved proceeding similarly as in the non-singular case.

\begin{theorem}\label{altcomp}
Let  $A = P_1-R_1+S_1$ be a double proper weak regular  splitting and $A = P_2-R_2+S_2$ be a double proper regular splitting of a semi-monotone matrix $A\in{\mathbb{R}^{m\times{n}}}$. If $N(S_2)\supseteq N(P_2)$, $R(S_2)\subseteq R(P_2)$, $1 \notin \sigma(S_2P_1^{\dag})$ and $\widehat{A}^{\dag}\geq 0$, then $\rho(W_{12})\leq \rho({T}_2)<1.$ 
\end{theorem}

The importance of the TG-ADS scheme is discussed in the next result which is a combination of Theorem  \ref{altcomp121} and  Theorem \ref{altcomp}.
The result says that the proposed ADS scheme \eqref{eq5}
converges faster than usual double iteration scheme  \eqref{eq3} under suitable assumptions. 

\begin{theorem}\label{Hybcomp}
Let  $A = P_1-R_1+S_1$ be a double proper  weak regular  splitting
and $A = P_2-R_2+S_2$ be a double proper  regular splitting of a  semi-monotone matrix $A\in{\mathbb{R}^{m\times{n}}}$. If $N(S_2)\supseteq N(P_2)$, $R(S_2)\subseteq R(P_2)$, $1 \notin \sigma(S_2P_1^{\dag})$,  $\widehat{A}^{\dag}\geq 0$,   $P_1^{\dag}R_1\geq P_2^{\dag}R_2$ and $P_2^{\dag}S_2\geq P_1^{\dag}S_1$, then $$\rho(W_{12})\leq \mbox{min}
\{\rho({T}_1), \rho({T}_2)\}<1.$$ 
\end{theorem}

The corollary obtained below is even new in the non-singular matrix setting. 

\begin{corollary}\label{Hybcompcor}
Let  $A = P_1-R_1+S_1$ be a double  weak regular  splitting
and $A = P_2-R_2+S_2$ be a double regular splitting of a  monotone matrix $A\in{\mathbb{R}^{n\times {n}}}$. If $1 \notin \sigma(S_2P_1^{-1})$,  $\widehat{A}^{-1}\geq 0$,   $P_1^{-1}R_1\geq P_2^{-1}R_2$ and $P_2^{-1}S_2\geq P_1^{-1}S_1$, then $$\rho(W_{12})\leq \mbox{min}
\{\rho({T}_1), \rho({T}_2)\}<1.$$ 
\end{corollary}

The next result provides different sufficient conditions to draw the conclusion of Theorem \ref{altcomp121}.

\begin{theorem}\label{altcomp122}
Let $A=P_{1}-R_{1}+S_{1}$ be a double proper  weak regular splitting and $A=P_{2}-R_{2}-S_{2}$ be a double proper regular splitting of a semi-monotone matrix $A\in \mathbb{R}^{m\times n}$. Suppose that $N(S_{2})\supseteq N(P_{2})$, $R(S_{2})\subseteq R(P_{2})$, $1 \notin \sigma(S_2P_1^{\dag})$ and $\widehat{A}^{\dagger}\geq 0$. If $P_{1}^{\dagger}\leq P_{2}^{\dagger}$ and $P_{1}^{\dagger}R_{1}\leq P_{2}^{\dagger}R_{2}$, then $\rho(W_{12})\leq \rho(T_{1})<1$.
\end{theorem}
\begin{proof}
By Theorem \ref{2.1.1} and Theorem \ref{conv2}, we have $\rho({T}_1)<1$ and $\rho(W_{12})<1$, respectively.\\
Case (i): $\rho(W_{12}) = 0$.  The proof is obvious.\\ 
Case (ii): $\rho(W_{12}) \neq 0$. Since ${\bf W}_{12}\geq 0$,  there exists a non-negative eigenvector
${\bf x} = \begin{pmatrix}
    x_1 \\
    x_2
  \end{pmatrix} $
such that ${\bf W}_{12}{\bf x} = \rho(W_{12}){\bf x}$, i.e.,  $\widehat{{\bf T}}{\bf x} = \rho(W_{12}){\bf x}$ by Theorem \ref{2.1.4}. This gives
\begin{eqnarray*}
\widehat{P}^{\dag}\widehat{R}x_1-\widehat{P}^{\dag}\widehat{S}x_2 &=& \rho(W_{12})x_1\\
x_1 &=& \rho(W_{12})x_2.
\end{eqnarray*}
Now 
\begin{eqnarray*}
{\bf T}_1{\bf x}-\rho(W_{12}){\bf x} &=& \begin{pmatrix}
    P_1^{\dag}R_1x_1-P_1^{\dag}S_1x_2-\rho(W_{12})x_1\\ 
    x_1-\rho(W_{12})x_2
    \end{pmatrix}.  
\end{eqnarray*}
The condition $P_{1}^{\dagger}R_{1}\leq P_{2}^{\dagger}R_{2}$ yields
$$P_{1}^{\dagger}R_{1}-P_{2}^{\dagger}R_{2}+P_{2}^{\dagger}S_{2}P_{1}^{\dagger}R_{1}\leq 0~~ \mbox{i.e.,}~~P_{1}^{\dagger}R_{1}-P_{2}^{\dagger}(R_{2}-S_{2}P_{1}^{\dagger}R_{1})\leq 0.$$
Hence $P_{1}^{\dagger}R_{1}-\widehat{P}^{\dagger}\widehat{R}\leq0.$
Now
\begin{eqnarray*}
    && P_{1}^{\dagger}R_{1}x_{1}-P_{1}^{\dagger}S_{1} x_{2}-\rho(W_{12})x_{1}-\frac{1}{\rho(W_{12})}(P_{1}^{\dagger}R_{1}-P_{1}^{\dagger}S_{1})x_{1}-\frac{1}{\rho(W_{12})}(\widehat{P}^{\dagger}\widehat{S}-\widehat{P}^{\dagger}\widehat{R})x_{1}\\
    &=& P_{1}^{\dagger}R_{1}x_{1}-\frac{1}{\rho(W_{12})}P_{1}^{\dagger}S_{1}x_{1}-\widehat{P}^{\dagger}\widehat{R}x_{1}+\frac{1}{\rho(W_{12})}\widehat{P}^{\dagger}\widehat{S}x_{1}\\
    &~&~~~~~~~~~~-\frac{1}{\rho(W_{12})}P_{1}^{\dagger}R_{1}x_{1}+\frac{1}{\rho(W_{12})}P_{1}^{\dagger}S_{1}x_{1}-\frac{1}{\rho(W_{12})}\widehat{P}^{\dagger}\widehat{S}x_{1}+\frac{1}{\rho(W_{12})}\widehat{P}^{\dagger}\widehat{R}x_{1}\\
    &=&\left(1-\frac{1}{\rho(W_{12})}\right)P_{1}^{\dagger}R_{1}x_{1}+\left(\frac{1}{\rho(W_{12})}-1\right)\widehat{P}^{\dagger}\widehat{R}x_{1}\\
    &=&\left(1-\frac{1}{\rho(W_{12})}\right)(P_{1}^{\dagger}R_{1}-\widehat{P}^{\dagger}\widehat{R})x_{1}\geq 0.
\end{eqnarray*}
Therefore, \begin{eqnarray*}
    P_{1}^{\dagger}R_{1}x_{1}-P_{1}^{\dagger}S_{1} x_{2}-\rho(W_{12})x_{1} &\geq& \frac{1}{\rho(W_{12})}(P_{1}^{\dagger}R_{1}-P_{1}^{\dagger}S_{1}+\widehat{P}^{\dagger}\widehat{S}-\widehat{P}^{\dagger}\widehat{R})x_{1}\\
    &=&\frac{1}{\rho(W_{12})}\left(P_{1}^{\dagger}(R_{1}-S_{1})+\widehat{P}^{\dagger}(\widehat{S}-\widehat{R})\right)x_{1}\\
    &=&\frac{1}{\rho(W_{12})}\left(P_{1}^{\dagger}(P_{1}-A)+\widehat{P}^{\dagger}(\widehat{A}-\widehat{P})\right)x_{1}\\
        &=&\frac{1}{\rho(W_{12})}(P_{2}^{\dagger}P_{2}+\widehat{P}^{\dagger}\widehat{A}-P_{1}^{\dagger}A-P_{2}^{\dagger}P_{2})x_{1}\\
    &=&\frac{1}{\rho(W_{12})}(P_{2}^{\dagger}A-P_{2}^{\dagger}S_{2}P_{1}^{\dagger}A-P_{1}^{\dagger}A)x_{1}\\
    &=&\frac{1}{\rho(W_{12})}(P_{2}^{\dagger}-P_{1}^{\dagger})Ax_{1}+\frac{1}{\rho(W_{12})}(-P_{2}^{\dagger}S_{2}P_{1}^{\dagger})Ax_1\geq0,
\end{eqnarray*}
as $Ax_1\geq 0$ can be shown as in the previous proof. We thus have ${\bf T}_1{\bf x}-\rho(W_{12}){\bf x}\geq 0$  resulting  $   \rho(W_{12})\leq \rho(T_{1})<1$
 by Theorem \ref{2.1.2}.
\end{proof}

\begin{corollary}
Let $A=P_{1}-R_{1}+S_{1}$ be a double  weak regular splitting and $A=P_{2}-R_{2}-S_{2}$ be a double  regular splitting of a monotone matrix $A\in \mathbb{R}^{n\times n}$. Suppose that $1 \notin \sigma(S_2P_1^{-1})$ and $\widehat{A}^{-1}\geq 0$. If $P_{1}^{-1}\leq P_{2}^{-1}$ and $P_{1}^{-1}R_{1}\leq P_{2}^{-1}R_{2}$, then $$\rho(W_{12})\leq \rho(T_{1})<1.$$
\end{corollary}

Note that the conclusion of Theorem \ref{Hybcomp} also follows if the conditions $P_1^{\dag}R_1\geq P_2^{\dag}R_2$ and $P_2^{\dag}S_2\geq P_1^{\dag}S_1$ are replaced by $P_{1}^{\dagger}\leq P_{2}^{\dagger}$ and $P_{1}^{\dagger}R_{1}\leq P_{2}^{\dagger}R_{2}$.
Based on the above-discussed results, it is confirmed that the TG-ADS scheme is a better choice for a certain class of matrices. However, if a matrix $A$ has many pairs of double proper splittings satisfying the desired convergence criteria for the TG-ADS scheme, we face another problem, i.e.,  if $A$ has three or more double proper splittings of $A\in{\mathbb{R}^{m\times{n}}}$, the problem is to choose which pair of double proper splittings to frame the TG hydrid scheme. And to do this, we present a few comparison results next.  

Let $A = P_1-R_1+S_1 = P_2-R_2+S_2 = P_3-R_3+S_3$ be three double proper splittings of $A\in{\mathbb{R}^{m\times{n}}}$. 
Then, the iteration matrices for framing ADS schemes as in \eqref{eqq1} and \eqref{eqq2} are   
\begin{equation*}
   {\bf G}_1  = \begin{pmatrix}
    I & 0\\
    P_1^{\dag}R_1 & -P_1^{\dag}S_1 
    \end{pmatrix},~ {\bf G}_2 = \begin{pmatrix}
    I & 0\\
    P_2^{\dag}R_2 & -P_2^{\dag}S_2 
    \end{pmatrix}
   \end{equation*}
and 
 \begin{equation*}
     {\bf T}_3  = \begin{pmatrix}
    P_3^{\dag}R_3 & -P_3^{\dag}S_3\\ 
    I & 0
    \end{pmatrix}.
 \end{equation*}
 
If $N(S_3)\supseteq N(P_3)$, $R(S_3)\subseteq R(P_3)$ and $1 \notin \sigma(S_3P_1^{\dag})$, then $A = P_1-R_1+S_1 = P_3-R_3+S_3$ induce a  double proper splitting $\widehat{A}_1 = \widehat{P}_1-\widehat{R}_1+\widehat{S}_1$  to solve the preconditioned linear system $\widehat{A}_1x = \widehat{b}_1$.  The iteration matrix corresponding the double iterative scheme \eqref{preeq1} is 
\begin{equation*}
{\bf W}_{13} =  {\bf T}_3{\bf G}_1 =  \begin{pmatrix}
    P_3^{\dag}R_3-P_3^{\dag}S_3P_1^{\dag}R_1 & P_3^{\dag}S_3P_1^{\dag}S_1\\ 
    I & 0
    \end{pmatrix}. 
\end{equation*}
Similarly, assuming $N(S_3)\supseteq N(P_3)$, $R(S_3)\subseteq R(P_3)$ and $1 \notin \sigma(S_3P_2^{\dag})$, the other pair of double proper splittings induces another double proper splitting   $\widehat{A}_2 = \widehat{P}_2-\widehat{R}_2+\widehat{S}_2$ to solve $\widehat{A}_2x = \widehat{b}_2$. The corresponding iteration matrix is   
\begin{equation*}
{\bf W}_{23} = {\bf T}_3{\bf G}_2 =  \begin{pmatrix}
    P_3^{\dag}R_3-P_3^{\dag}S_3P_2^{\dag}R_2 & P_3^{\dag}S_3P_2^{\dag}S_2\\ 
    I & 0
    \end{pmatrix}.    
\end{equation*}
 
Next result presents a comparison result between the spectral radii of ${\bf W}_{13}$ and ${\bf W}_{23}$ which will help to know which pair of double splittings yields a better ADS scheme.

\begin{theorem}\label{dcomp}
Let $A = P_1-R_1+S_1 = P_2-R_2+S_2$ be two double proper weak regular  splittings of $A\in{\mathbb{R}^{m\times{n}}}$. Suppose that $A = P_3-R_3+S_3$ is a double proper regular splitting with $N(S_3) \supseteq N(P_3)$, $R(S_3) \subseteq R(P_3)$, $1 \notin \sigma(S_3P_i^{\dag})$ and  $\widehat{A}_i^{\dag}\widehat{P}_i\geq 0$ for $i = 1,2$. If $P_1^{\dag}\geq P_2^{\dag}$ and one of the following conditions \\
$(1)$ $P_1^{\dag}R_1\geq P_2^{\dag}R_2$\\
$(2)$ $P_1^{\dag}S_1\geq P_2^{\dag}S_2$\\
    holds, then $\rho(W_{13})\leq \rho(W_{23})<1$.
\end{theorem}
\begin{proof}
Clearly, we have $\rho(W_{13})<1$ and $\rho(W_{23})<1$, by Theorem \ref{conv3}.\\
Case (i): $\rho(W_{13}) = 0$.  The proof is obvious.\\ 
Case (ii): $\rho(W_{13}) \neq 0$. Since ${\bf W}_{13}\geq 0$, there exists a non-negative eigenvector
${\bf x} = \begin{pmatrix}
    x_1 \\
    x_2
\end{pmatrix} $
such that ${\bf W}_{13}{\bf x} = \rho(W_{13}){\bf x}$ by Theorem \ref{2.1.4}. This implies
\begin{align*}
 (P_3^{\dag}R_3-P_3^{\dag}S_3P_1^{\dag}R_1)x_1+P_3^{\dag}S_3P_1^{\dag}S_1x_2 = \rho(W_{13})x_1\\  
  x_1 = \rho(W_{13})x_2.
\end{align*}
As an immediate consequence, we have
\begin{eqnarray*}
{\bf W}_{23}{\bf x}-\rho(W_{13}){\bf x} &=& \begin{pmatrix}
    (P_3^{\dag}R_3-P_3^{\dag}S_3P_2^{\dag}R_2)x_1+P_3^{\dag}S_3P_2^{\dag}S_2x_2-\rho(W_{13})x_1\\ 
    x_1-\rho(W_{13})x_2
    \end{pmatrix}\\  
&=& \begin{pmatrix}
    -P_3^{\dag}S_3P_2^{\dag}R_2x_1+\displaystyle{\frac{1}{\rho( W_{13})}}P_3^{\dag}S_3P_2^{\dag}S_2x_1+P_3^{\dag}S_3P_1^{\dag}R_1x_1-\displaystyle{\frac{1}{\rho (W_{13})}}P_3^{\dag}S_3P_1^{\dag}S_1x_1\\ 
    0
    \end{pmatrix}\\
   &=& \begin{pmatrix}
   \nabla\\ 
    0
    \end{pmatrix},
\end{eqnarray*}
where \begin{equation*}\label{eq22}
   \nabla = P_3^{\dag}S_3(P_1^{\dag}R_1-P_2^{\dag}R_2)x_1+\displaystyle{\frac{1}{\rho(W_{13})}}P_3^{\dag}S_3(P_2^{\dag}S_2-P_1^{\dag}S_1)x_1. 
\end{equation*}
If the first condition $P_1^{\dag}R_1\geq P_2^{\dag}R_2$ holds, we then have
\begin{eqnarray*}
&& \nabla -\frac{1}{\rho(W_{13})}P_3^{\dag}S_3(P_1^{\dag}R_1-P_2^{\dag}R_2)x_1-\frac{1}{\rho(W_{13})}P_3^{\dag}S_3(P_2^{\dag}S_2-P_1^{\dag}S_1)x_1\\
&=& P_3^{\dag}S_3(P_1^{\dag}R_1-P_2^{\dag}R_2)x_1-\frac{1}{\rho(W_{13})}P_3^{\dag}S_3(P_1^{\dag}R_1-P_2^{\dag}R_2)x_1\\
&=& \left(1-\frac{1}{\rho(W_{13})}\right)P_3^{\dag}S_3(P_1^{\dag}R_1-P_2^{\dag}R_2)x_1 \geq 0.
\end{eqnarray*}
\noindent Therefore,
\begin{eqnarray*}
\nabla &\geq& \frac{1}{\rho(W_{13})}P_3^{\dag}S_3(P_1^{\dag}R_1-P_2^{\dag}R_2)x_1+\frac{1}{\rho(W_{13})}P_3^{\dag}S_3(P_2^{\dag}S_2-P_1^{\dag}S_1)x_1\\
&=& \frac{1}{\rho(W_{13})}P_3^{\dag}S_3\left(P_1^{\dag}R_1-P_1^{\dag}S_1+P_2^{\dag}S_2-P_2^{\dag}R_2\right)x_1\\
&=& \frac{1}{\rho(W_{13})}P_3^{\dag}S_3\left(P_1^{\dag}(R_1-S_1)+P_2^{\dag}(S_2-R_2)\right)x_1\\
&=& \frac{1}{\rho(W_{13})}P_3^{\dag}S_3\left(P_1^{\dag}(P_1-A)+P_2^{\dag}(A-P_2)\right)x_1\\
&=& \frac{1}{\rho(W_{13})}P_3^{\dag}S_3(P_2^{\dag}A-P_1^{\dag}A)x_1 ~~~~~~~~~~~~~~~~~~~~(\because P_2^{\dag}P_2=P_1^{\dag}P_1)\\
&=& \frac{1}{\rho(W_{13})}P_3^{\dag}S_3(P_2^{\dag}-P_1^{\dag})Ax_1 \geq 0,
\end{eqnarray*}
using the fact $Ax_1\geq 0$  as shown in Theorem \ref{altcomp121}.
Hence ${\bf W}_{23}{\bf x}-\rho(W_{13}){\bf x} \geq 0$. By Theorem \ref{2.1.2}, $\rho(W_{13})\leq \rho(W_{23})$.\\
Similarly, if $P_1^{\dag}S_1\geq P_2^{\dag}S_2$, one can easily show that
$$ \nabla-P_3^{\dag}S_3(P_1^{\dag}R_1-P_2^{\dag}R_2)x_1-P_3^{\dag}S_3(P_2^{\dag}S_2-P_1^{\dag}S_1)x_1\geq 0$$
i.e., 
$$\nabla  \geq P_3^{\dag}S_3(P_1^{\dag}R_1-P_2^{\dag}R_2)x_1+P_3^{\dag}S_3(P_2^{\dag}S_2-P_1^{\dag}S_1)x_1 \geq 0.
$$
Thus, ${\bf W}_{23}{\bf x}-\rho(W_{13}){\bf x} \geq 0$. By Theorem \ref{2.1.2}, it follows that $\rho(W_{13})\leq \rho(W_{23})$.

\end{proof}

Note that the condition $\widehat{A}_i^{\dag}\widehat{P}_i\geq 0$ for $i = 1,2$ is assumed in the above theorem as the class of matrices ($\widehat{A}^{\dag}_i\widehat{P}_i\geq 0$) is bigger than the class $\widehat{A}_i^{\dag}\geq 0$ for $i = 1,2$ and each double proper regular (weak) splitting is also a double proper weak splitting. The above theorem is also true if we replace the condition $P_1^{\dag}\geq P_2^{\dag}$ by $P_1^{\dag}A\geq P_2^{\dag}A$.
We have the following corollary to the above result in the case of non-singular matrix setting.

\begin{corollary}[Theorem $3.2$, \cite{li}]
Let $A = P_1-R_1+S_1 = P_2-R_2+S_2$ be two double weak regular  splittings and $A = P_3-R_3+S_3$ be a double regular splitting of a  real non-singular matrix $A$. Suppose that $1 \notin \sigma(S_3P_i^{-1})$ and  $\widehat{A}_i^{-1}\widehat{P}_i\geq 0$ for $i = 1,2$. If $P_1^{-1}\geq P_2^{-1}$ and one of the following conditions \\
$(1)$ $P_1^{-1}R_1\geq P_2^{-1}R_2$\\
$(2)$ $P_1^{-1}S_1\geq P_2^{-1}S_2$\\
holds, then $\rho(W_{13})\leq \rho(W_{23})<1$.
\end{corollary}

We have the following comparison result between the spectral radii of the iteration matrices ${\bf W}_{13}$ and ${\bf W}_{23}$ which is motivated by the proof of Theorem 3.7, \cite{jmp}.

\begin{theorem}\label{dcomp3}
Let $A = P_1-R_1+S_1 = P_2-R_2+S_2$ be two double proper weak regular  splittings of a semi-monotone matrix $A\in{\mathbb{R}^{m\times{n}}}$. Suppose that $A = P_3-R_3+S_3$ is a  double proper regular splitting with $N(S_3) \supseteq N(P_3)$, $R(S_3) \subseteq R(P_3)$, $1 \notin \sigma(S_3P_i^{\dag})$ and $\widehat{A}_i^{\dag}\geq 0$ for $i = 1,2$. If $\widehat{P}_1^{\dag}\widehat{A}_1\geq \widehat{P}_2^{\dag}\widehat{A}_2$ and $\widehat{P}_1^{\dag}\widehat{R}_1\geq \widehat{P}_2^{\dag}\widehat{R}_2$, then $\rho(W_{13})\leq \rho(W_{23})<1$.
\end{theorem}
\begin{proof}
We have $\widehat{A}_1 = (I-S_3P_1^{\dag})A$, $\widehat{P}_1 = P_3$, $\widehat{R}_1 = R_3-S_3P_1^{\dag}R_1$ and $\widehat{S}_1 = -S_3P_1^{\dag}S_1$ corresponding to the two  double proper weak regular splittings $A = P_1-R_1+S_1 = P_3-R_3+S_3$. Again, the double  proper weak regular splittings $A = P_2-R_2+S_2 = P_3-R_3+S_3$ give $\widehat{A}_2 = (I-S_3P_2^{\dag})A$, $\widehat{P}_2 = P_3$, $\widehat{R}_2 = R_3-S_3P_2^{\dag}R_2$ and $\widehat{S}_2 = -S_3P_2^{\dag}S_2$. The respective iteration matrices are
\begin{equation*}
  \widehat{\bf T}_1 = \begin{pmatrix}
    \widehat{P}_1^{\dag}\widehat{R}_1 & -\widehat{P}_1^{\dag}\widehat{S}_1 \\
    I & 0
  \end{pmatrix}={\bf W}_{13},~ \widehat{\bf T}_2 = \begin{pmatrix}
    \widehat{P}_2^{\dag}\widehat{R}_2 & -\widehat{P}_2^{\dag}\widehat{S}_2 \\
    I & 0
  \end{pmatrix}={\bf W}_{23}.
\end{equation*}
By Theorem \ref{conv2}, we have $\rho(W_{13})<1$ and $\rho(W_{23})<1$. \\
Case (i): $\rho(W_{13}) = 0$.  The proof is obvious.\\ 
Case (ii): $\rho(W_{13}) \neq 0$.
Since ${\bf W}_{13}\geq 0$,  there exists a non-negative eigenvector
${\bf x} = \begin{pmatrix}
    x_1 \\
    x_2
\end{pmatrix} $
such that ${\bf W}_{13}{\bf x} = \rho(W_{13}){\bf x}$ by Theorem \ref{2.1.4}.
   This yields
\begin{align*}
 \widehat{P}_1^{\dag}\widehat{R}_1x_1-\widehat{P}_1^{\dag}\widehat{S}_1x_2 = \rho(W_{13})x_1\\  
  x_1 = \rho(W_{13})x_2. 
\end{align*}
By using the above two equations, we have
\begin{eqnarray*}
{\bf W}_{23}{\bf x}-\rho(W_{13}){\bf x} &=& \begin{pmatrix}
    \widehat{P}_2^{\dag}\widehat{R}_2 x_1-\widehat{P}_2^{\dag}\widehat{S}_2x_2-\rho(W_{13})x_1\\ 
    x_1-\rho(W_{13})x_2
    \end{pmatrix}\\  
&=& \begin{pmatrix}
    \widehat{P}_2^{\dag}\widehat{R}_2 x_1-\widehat{P}_2^{\dag}\widehat{S}_2x_2-\widehat{P}_1^{\dag}\widehat{R}_1x_1+\widehat{P}_1^{\dag}\widehat{S}_1x_2\\ 
    0
    \end{pmatrix}\\
 &=& \begin{pmatrix}
    \widehat{P}_2^{\dag}\widehat{R}_2 x_1-\widehat{P}_1^{\dag}\widehat{R}_1x_1-\displaystyle{\frac{1}{\rho(W_{13})}}(\widehat{P}_2^{\dag}\widehat{S}_2-\widehat{P}_1^{\dag}\widehat{S}_1)x_1\\ 
    0
    \end{pmatrix}\\
  &\geq& \begin{pmatrix}
    \displaystyle{\frac{1}{\rho(W_{13})}}(\widehat{P}_2^{\dag}\widehat{R}_2-\widehat{P}_1^{\dag}\widehat{R}_1-\widehat{P}_2^{\dag}\widehat{S}_2+\widehat{P}_1^{\dag}\widehat{S}_1)x_1\\ 
    0
    \end{pmatrix}\\  
   &=& \begin{pmatrix}
    \displaystyle{\frac{1}{\rho(W_{13})}}\left(\widehat{P}_2^{\dag}(\widehat{R}_2-\widehat{S}_2)-\widehat{P}_1^{\dag}(\widehat{R}_1-\widehat{S}_1)\right)x_1\\ 
    0
    \end{pmatrix}\\
    &=& \begin{pmatrix}
    \displaystyle{\frac{1}{\rho(W_{13})}}(P_3^{\dag}P_3-\widehat{P}_2^{\dag}\widehat{A}_2-P_3^{\dag}P_3+\widehat{P}_1^{\dag}\widehat{A}_1)x_1\\ 
    0
    \end{pmatrix}\\ 
    &=& \begin{pmatrix}
    \displaystyle{\frac{1}{\rho(W_{13})}}(\widehat{P}_1^{\dag}\widehat{A}_1-\widehat{P}_2^{\dag}\widehat{A}_2)x_1\\ 
    0
    \end{pmatrix}\geq 0.
\end{eqnarray*}
 Hence, the conclusion follows  by Theorem \ref{2.1.2}.
\end{proof}

\begin{corollary}
Let $A = P_1-R_1+S_1 = P_2-R_2+S_2$ be two double weak regular  splittings of a monotone matrix $A\in{\mathbb{R}^{n\times{n}}}$.  Suppose that $A = P_3-R_3+S_3$ be a double regular  splitting with $1 \notin \sigma(S_3P_i^{-1})$ and  $\widehat{A}_i^{-1}\geq 0$ for $i = 1,2$. If $\widehat{P}_1^{-1}\widehat{A}_1\geq \widehat{P}_2^{-1}\widehat{A}_2$ and $\widehat{P}_1^{-1}\widehat{R}_1\geq \widehat{P}_2^{-1}\widehat{R}_2$, then $\rho(W_{13})\leq \rho(W_{23})<1$.
\end{corollary}

\noindent Observe that the above comparisons are made between two pairs of double proper splittings with one common double proper  splitting out of three independent double proper splittings. Further, we are interested to reveal the comparison of two independent pairs of double proper splittings tailored by the TG-ADS scheme \eqref{eq5}. 
To this end, let us consider the first pair of double proper splittings $A = 
P_1-R_1+S_1 = P_2-R_2+S_2$ such that $N(S_2)\supseteq N(P_2)$, $R(S_2)\subseteq R(P_2)$ and $1 \notin \sigma(S_2P_1^{\dag})$. As per the convergence analysis described in the subsection \ref{sec:cgs}, the corresponding induced splitting $\widehat{A}_1 = \widehat{P}_1-\widehat{R}_1+\widehat{S}_1$ is a double proper  splitting, where $\widehat{A}_1 = (I-S_2P_1^{\dag})A$, $\widehat{P}_1 = P_2$, $\widehat{R}_1 = R_2-S_2P_1^{\dag}R_1$ and $\widehat{S}_1 = -S_2P_1^{\dag}S_1$. In this case, the iteration matrix of the TG-ADS scheme  is 
\begin{equation*}
 {\bf W}_{12} =  \begin{pmatrix}
    P_2^{\dag}R_2-P_2^{\dag}S_2P_1^{\dag}R_1 & P_2^{\dag}S_2P_1^{\dag}S_1\\ 
    I & 0
    \end{pmatrix}
    \end{equation*}
    which is also the iteration matrix of double iteration scheme  \eqref{preeq1} to solve the system $\widehat{A}_1x = \widehat{b}_1$.
Similarly, the other pair of double proper  splittings 
$A = P_3-R_3+S_3 = P_4-R_4+S_4$, satisfying $N(S_4)\supseteq N(P_4)$, $R(S_4)\subseteq R(P_4)$ and $1 \notin \sigma(S_4P_3^{\dag})$ yields 
   
\begin{equation*}
{\bf W}_{34} =  \begin{pmatrix}
    P_4^{\dag}R_4-P_4^{\dag}S_4P_3^{\dag}R_3 & P_4^{\dag}S_4P_3^{\dag}S_3\\ 
    I & 0
    \end{pmatrix}. 
\end{equation*}
The following theorem presents different sufficient conditions for choosing a better ADS scheme, and its proof is analogous to the proof of Theorem \ref{dcomp3}. 

\begin{theorem}\label{wcomp1}
Let $A = P_1-R_1+S_1 = P_2-R_2+S_2 = P_3-R_3+S_3 = P_4-R_4+S_4$ be four double  proper weak regular  splittings of a semi-monotone matrix $A\in{\mathbb{R}^{m\times{n}}}$. Suppose that $\widehat{A}_i^{\dag}\geq 0$ for $i = 1,2$, $N(S_2) \supseteq N(P_2)$, $R(S_2) \subseteq R(P_2)$, $N(S_4) \supseteq N(P_4)$, $R(S_4) \subseteq R(P_4)$, $1 \notin \sigma(S_2P_1^{\dag})$ and $1 \notin \sigma(S_4P_3^{\dag})$.  If $\widehat{P}_1^{\dag}\widehat{A}_1\geq \widehat{P}_2^{\dag}\widehat{A}_2$ and $\widehat{P}_1^{\dag}\widehat{R}_1\geq \widehat{P}_2^{\dag}\widehat{R}_2$, then $\rho(W_{12})\leq \rho(W_{34})<1$.
\end{theorem}

As a consequence, we have the next result. 
\begin{corollary}
Let $A = P_1-R_1+S_1 = P_2-R_2+S_2 = P_3-R_3+S_3 = 
P_4-R_4+S_4$ be four double weak regular  splittings of a monotone matrix $A\in{\mathbb{R}^{n\times{n}}}$. Suppose that $\widehat{A}_i^{-1}\geq 0$ for $i = 1,2$, $1 \notin \sigma(S_2P_1^{-1})$ and $1 \notin \sigma(S_4P_3^{-1})$.  If $\widehat{P}_1^{-1}\widehat{A}_1\geq \widehat{P}_2^{-1}\widehat{A}_2$ and $\widehat{P}_1^{-1}\widehat{R}_1\geq \widehat{P}_2^{-1}\widehat{R}_2$, then $\rho(W_{12})\leq \rho(W_{34})<1$.
\end{corollary}

\subsection{Comparison Results: HT-ADS scheme}
In this sub-section, we  establish that the HT-ADS scheme converges faster than the classical double spitting schemes. Further, we answer the natural question that under what condition the HT-ADS scheme performs better than the TG-ADS scheme and vice versa. The comparisons among $\rho({\mathcal{W}}_{12}), \rho({\mathcal{W}}_{23}) \mbox{ and } \rho({\mathcal{W}}_{34})$, like the TG-ADS scheme are omitted as they are very similar.

\begin{theorem}\label{wt2comp2}
Let $A = P_1-R_1+S_1$ be a double proper weak regular splitting and $A = P_2-R_2+S_2$ be a double proper  regular splitting of a semi-monotone matrix   $A\in{\mathbb{R}^{m\times{n}}}$. Suppose that $N(R_2) \supseteq N(P_2)$, $R(R_2) \subseteq R(P_2)$, $-1 \notin \sigma(R_2P_1^{\dag})$ and $\widehat{\mathcal{A}}^{\dag}\geq 0$. If $P_2^{\dag}R_2+P_2^{\dag}S_2\leq 0$, then $\rho({\mathcal{W}}_{12})\leq\rho(T_2)<1.$
\end{theorem}

\begin{proof}
By Theorem \ref{wwcon} and Theorem \ref{2.1.1}, we have $\rho({\mathcal{W}}_{12})<1$ and $\rho({T}_2)<1$, respectively. \\
Case (i):  $\rho({\mathcal{W}}_{12}) = 0$, the proof is trivial. \\
Case (ii):  $\rho({\mathcal{W}}_{12})\neq 0$, i.e., $0<\rho({\mathcal{W}}_{12})<1$.
Applying Theorem \ref{2.1.4} to ${\bf {\WW}_{12}}$, i.e., there exists a vector 
$${\bf x} = \begin{pmatrix}
    x_1 \\
    x_2
  \end{pmatrix} \geq 0,~ {\bf x}\neq 0$$
such that ${\bf {\WW}_{12}}{\bf x} = \rho({\mathcal{W}}_{12}){\bf x}$. This gives
\begin{eqnarray*}
(P_2^{\dag}R_2P_1^{\dag}R_1-P_2^{\dag}S_2)x_1-P_2^{\dag}R_2P_1^{\dag}S_1x_2 &=& \rho({\mathcal{W}}_{12})x_1\\
x_1 &=& \rho({\mathcal{W}}_{12})x_2.
\end{eqnarray*}
So, we have
\begin{eqnarray*}
{\bf T}_2{\bf x}-\rho({\mathcal{W}}_{12}){\bf x} &=& \begin{pmatrix}
    P_2^{\dag}R_2x_1-P_2^{\dag}S_2x_2-\rho({\mathcal{W}}_{12})x_1\\ 
    x_1-\rho({\mathcal{W}}_{12})x_2
    \end{pmatrix}.
\end{eqnarray*}
By suitable substitutions in the first component of the above expression, we have
\begin{eqnarray*}
&& P_2^{\dag}R_2x_1-\frac{1}{\rho({\mathcal{W}}_{12})}P_2^{\dag}S_2x_1-P_2^{\dag}R_2P_1^{\dag}R_1x_1+P_2^{\dag}S_2x_1+\frac{1}{\rho({\mathcal{W}}_{12})}P_2^{\dag}R_2P_1^{\dag}S_1x_1\\
&\geq& P_2^{\dag}R_2x_1+P_2^{\dag}S_2x_1-\frac{1}{\rho({\mathcal{W}}_{12})}P_2^{\dag}S_2x_1-\frac{1}{\rho({\mathcal{W}}_{12})}P_2^{\dag}R_2P_1^{\dag}R_1x_1+\frac{1}{\rho({\mathcal{W}}_{12})}P_2^{\dag}R_2P_1^{\dag}S_1x_1\\
&=& P_2^{\dag}R_2x_1+\left(1-\frac{1}{\rho({\mathcal{W}}_{12})}\right)P_2^{\dag}S_2x_1+\frac{1}{\rho({\mathcal{W}}_{12})}P_2^{\dag}R_2P_1^{\dag}(-R_1+S_1)x_1\\
&=& P_2^{\dag}R_2x_1+\left(1-\frac{1}{\rho({\mathcal{W}}_{12})}\right)P_2^{\dag}S_2x_1+\frac{1}{\rho({\mathcal{W}}_{12})}P_2^{\dag}R_2P_1^{\dag}(A-P_1)x_1\\
&=& P_2^{\dag}R_2x_1+\left(1-\frac{1}{\rho({\mathcal{W}}_{12})}\right)P_2^{\dag}S_2x_1+\frac{1}{\rho({\mathcal{W}}_{12})}P_2^{\dag}R_2P_1^{\dag}Ax_1-\frac{1}{\rho({\mathcal{W}}_{12})}P_2^{\dag}R_2x_1\\
&\geq& \left(1-\frac{1}{\rho({\mathcal{W}}_{12})}\right)P_2^{\dag}R_2x_1 +\left(1-\frac{1}{\rho({\mathcal{W}}_{12})}\right)P_2^{\dag}S_2x_1\\ &~&~~~~~~~~~~~~~~~~~~~~~~~~~~~~~~~~~~~~~~~~~(\because P_2^{\dag}R_2P_1^{\dag}Ax_1\geq 0~ \mbox{as}~ Ax_1\geq 0~\mbox{by Theorem \ref{dcomp}})\\
&=& (1-\frac{1}{\rho({\mathcal{W}}_{12})})(P_2^{\dag}R_2+P_2^{\dag}S_2)\geq 0.
\end{eqnarray*}
 Hence ${\bf T}_2{\bf x}-\rho({\mathcal{W}}_{12}){\bf x} \geq 0$.
 We thus have $\rho({\mathcal{W}}_{12})\leq \rho({T}_2)<1,$ by Theorem \ref{2.1.2}.
\end{proof}

\begin{theorem}\label{wt1comp2}
Let $A = P_1-R_1+S_1$ be a double proper weak regular splitting and $A = P_2-R_2+S_2$ be a double proper regular splitting of a semi-monotone matrix   $A\in{\mathbb{R}^{m\times{n}}}$. Suppose that $N(R_2) \supseteq N(P_2)$, $R(R_2) \subseteq R(P_2)$, $-1 \notin \sigma(R_2P_1^{\dag})$ and $\widehat{\mathcal{A}}^{\dag}\geq 0$. If  $P_1^{\dag}R_1\geq P_2^{\dag}R_2$, $P_2^{\dag}S_2\geq P_1^{\dag}S_1$ and  $P_2^{\dag}R_2+P_2^{\dag}S_2\leq 0$, then $\rho({\mathcal{W}}_{12})\leq\rho(T_1)<1.$
\end{theorem}

\begin{proof}
By Theorem \ref{wwcon} and Theorem \ref{2.1.1}, we get $\rho({\mathcal{W}}_{12})<1$ and $\rho({T}_1)<1$, respectively.\\
Case (i):  $\rho({\mathcal{W}}_{12}) = 0$, the proof is trivial.\\
Case (ii): $\rho({\mathcal{W}}_{12})\neq 0$, i.e., $0<\rho({\mathcal{W}}_{12})<1$.
Applying Theorem \ref{2.1.4} to ${\bf {\WW}_{12}}$, i.e., there exists a vector 
$${\bf x} = \begin{pmatrix}
    x_1 \\
    x_2
  \end{pmatrix} \geq 0,~ {\bf x}\neq 0$$
such that ${\bf {\WW}_{12}}{\bf x} = \rho({\mathcal{W}}_{12}){\bf x}$. This gives
\begin{eqnarray*}
(P_2^{\dag}R_2P_1^{\dag}R_1-P_2^{\dag}S_2)x_1-P_2^{\dag}R_2P_1^{\dag}S_1x_2 &=& \rho({\mathcal{W}}_{12})x_1\\
x_1 &=& \rho({\mathcal{W}}_{12})x_2.
\end{eqnarray*}
Now 
\begin{eqnarray*}
{\bf T}_1{\bf x}-\rho({\mathcal{W}}_{12}){\bf x} &=& \begin{pmatrix}
    P_1^{\dag}R_1x_1-P_1^{\dag}S_1x_2-\rho({\mathcal{W}}_{12})x_1\\ 
    x_1-\rho({\mathcal{W}}_{12})x_2
    \end{pmatrix}.
\end{eqnarray*}
By suitable substitutions in the first component of the above expression, we have
\begin{eqnarray*}
&& P_1^{\dag}R_1x_1-\frac{1}{\rho({\mathcal{W}}_{12})}P_1^{\dag}S_1x_1-P_2^{\dag}R_2P_1^{\dag}R_1x_1+P_2^{\dag}S_2x_1+\frac{1}{\rho({\mathcal{W}}_{12})}P_2^{\dag}R_2P_1^{\dag}S_1x_1\\
&\geq& P_2^{\dag}R_2x_1-\frac{1}{\rho({\mathcal{W}}_{12})}P_1^{\dag}S_1x_1-\frac{1}{\rho({\mathcal{W}}_{12})}P_2^{\dag}R_2P_1^{\dag}R_1x_1+\frac{1}{\rho({\mathcal{W}}_{12})}P_2^{\dag}R_2P_1^{\dag}S_1x_1+P_2^{\dag}S_2x_1\\
&\geq& P_2^{\dag}R_2x_1+\left(1-\frac{1}{\rho({\mathcal{W}}_{12})}\right)P_2^{\dag}S_2x_1+\frac{1}{\rho({\mathcal{W}}_{12})}P_2^{\dag}R_2P_1^{\dag}(-R_1+S_1)x_1\\
&=& P_2^{\dag}R_2x_1+\left(1-\frac{1}{\rho({\mathcal{W}}_{12})}\right)P_2^{\dag}S_2x_1+\frac{1}{\rho({\mathcal{W}}_{12})}P_2^{\dag}R_2P_1^{\dag}(A-P_1)x_1\\
&=& P_2^{\dag}R_2x_1+\left(1-\frac{1}{\rho({\mathcal{W}}_{12})}\right)P_2^{\dag}S_2x_1+\frac{1}{\rho({\mathcal{W}}_{12})}P_2^{\dag}R_2P_1^{\dag}Ax_1-\frac{1}{\rho({\mathcal{W}}_{12})}P_2^{\dag}R_2x_1\\
&\geq& \left(1-\frac{1}{\rho({\mathcal{W}}_{12})}\right)P_2^{\dag}R_2x_1 +\left(1-\frac{1}{\rho({\mathcal{W}}_{12})}\right)P_2^{\dag}S_2x_1 ~~~(\because P_2^{\dag}R_2P_1^{\dag}Ax_1\geq 0)\\
&=& \left(1-\frac{1}{\rho({\mathcal{W}}_{12})}\right)(P_2^{\dag}R_2+P_2^{\dag}S_2)\geq 0.
\end{eqnarray*}
 Hence ${\bf T}_1{\bf x}-\rho({\mathcal{W}}_{12}){\bf x} \geq 0$. By Theorem \ref{2.1.2}, we have $\rho({\mathcal{W}}_{12})\leq \rho({T}_1)<1.$
\end{proof}

Combining the above two results, we have the following one.
\begin{theorem}\label{wt12comp2}
Let $A = P_1-R_1+S_1$ be a double proper weak regular splitting and $A = P_2-R_2+S_2$ be a double proper regular splitting of a semi-monotone matrix   $A\in{\mathbb{R}^{m\times{n}}}$. Suppose that $N(R_2) \supseteq N(P_2)$, $R(R_2) \subseteq R(P_2)$, $-1 \notin \sigma(R_2P_1^{\dag})$ and $\widehat{\mathcal{A}}^{\dag}\geq 0$. If  $P_1^{\dag}R_1\geq P_2^{\dag}R_2$, $P_2^{\dag}S_2\geq P_1^{\dag}S_1$ and  $P_2^{\dag}R_2+P_2^{\dag}S_2\leq 0$, then $\rho({\mathcal{W}}_{12})\leq min\{\rho(T_1), \rho(T_2)\} <1.$
\end{theorem}
In the case of non-singular $A$, we obtain the following result as a corollary. 
\begin{corollary}\label{wt12comp2cor}
Let $A = P_1-R_1+S_1$ be a double  weak regular splitting and $A = P_2-R_2+S_2$ be a double regular splitting of a monotone matrix   $A$. Suppose that $-1 \notin \sigma(R_2P_1^{-1})$ and $\widehat{\mathcal{A}}^{-1}\geq 0$. If  $P_1^{-1}R_1\geq P_2^{-1}R_2$, $P_2^{-1}S_2\geq P_1^{-1}S_1$ and  $P_2^{-1}R_2+P_2^{-1}S_2\leq 0$, then $\rho({\mathcal{W}}_{12})\leq min\{\rho(T_1), \rho(T_2)\} <1.$
\end{corollary}

In support of Theorem \ref{wt12comp2}, the following example is illustrated.

\begin{example}\label{exwt123}
Suppose $ A = \begin{bmatrix}
1 & 0 & 0 \\
0 & 0 & 1 
\end{bmatrix} = P_1-R_1+S_1 = P_2-R_2+S_2$   

\begin{align*}
=
\begin{bmatrix}
5 & 0 & 0 \\
0 & 0 & 5
\end{bmatrix}
-
\begin{bmatrix}
2 & 0 & 0 \\
0 & 0 & 2
\end{bmatrix}
+
\begin{bmatrix}
-2 & 0 & 0 \\
0 & 0 & -2
\end{bmatrix}\\
=\begin{bmatrix}
3 & 0 & 0 \\
0 & 0 & 3
\end{bmatrix}-\begin{bmatrix}
1 & 0 & 0 \\
0 & 0 & 1
\end{bmatrix}
+
\begin{bmatrix}
-1 & 0 & 0 \\
0 & 0 & -1
\end{bmatrix}
\end{align*}
are two double proper splittings of $A$ such that the first one is double proper weak regular splitting and the second one is double proper regular splitting. Also, it satisfy all conditions of Theorem \ref{wt12comp2}. Therefore,  $ \rho({\mathcal{W}}_{12}) = 0.6667 \leq min\{\rho(T_1) = 0.8633, \rho(T_2) = 0.7675\} <1.$
\end{example}

\begin{theorem}\label{wwcomp2}
Let $A = P_1-R_1+S_1$ be a double proper weak regular  splitting
and $A = P_2-R_2+S_2$ be a double proper regular splitting of a semi-monotone matrix $A\in{\mathbb{R}^{m\times{n}}}$. Suppose that  $N(S_2) = N(P_2)$, $R(S_2) = R(P_2)$, $1 \notin \sigma(S_2P_1^{\dag})$, $-1 \notin \sigma(R_2P_1^{\dag})$, $\widehat{A}^{\dag}\geq 0$ and $\widehat{\mathcal{A}}^{\dag}\geq 0$. If $\widehat{\mathcal{P}}^{\dag}\widehat{\mathcal{R}}\geq \widehat{P}^{\dag}\widehat{R}$ and $ \widehat{\mathcal{P}}^{\dag}\widehat{\mathcal{A}}\geq \widehat{P}^{\dag}\widehat{A}$, then $\rho({\mathcal{W}}_{12})\leq\rho({W}_{12})<1.$
\end{theorem}

\begin{proof}
As in the earlier proof, we get  $\widehat{A} = \widehat{P}-\widehat{R}+\widehat{S}$ and $\widehat{\mathcal{A}} = \widehat{\mathcal{P}}-\widehat{\mathcal{R}}+\widehat{\mathcal{S}}$ as double proper  weak regular splittings.
By Theorem \ref{conv2} and Theorem \ref{wwcon}, we have  $\rho(W_{12})<1$ and $\rho({\mathcal{W}}_{12})<1,$ respectively.\\
Case (i): $\rho({\mathcal{W}}_{12}) = 0$.  The proof is obvious.\\ 
Case (ii): $\rho({\mathcal{W}}_{12}) \neq 0$. Since ${\bf {\bf \WW}_{12}}\geq 0$,  there exists a non-negative eigenvector
${\bf x} = \begin{pmatrix}
    x_1 \\
    x_2
  \end{pmatrix} $
such that ${\bf {\bf \WW}_{12}}{\bf x} = \rho({\mathcal{W}}_{12}){\bf x}$ by Theorem \ref{2.1.4}, i.e.,
\begin{align*}
 \widehat{\mathcal{P}}^{\dag}\widehat{\mathcal{R}}x_1-\widehat{\mathcal{P}}^{\dag}\widehat{\mathcal{S}}x_2 = \rho({\mathcal{W}}_{12})x_1\\  
  x_1 = \rho({\mathcal{W}}_{12})x_2. 
\end{align*}
Now
\begin{eqnarray*}
{\bf W_{12}}{\bf x}-\rho({\mathcal{W}}_{12}){\bf x} &=& \begin{pmatrix}
    \widehat{P}^{\dag}\widehat{R} x_1-\widehat{P}^{\dag}\widehat{S}x_2-\rho({\mathcal{W}}_{12})x_1\\ 
    x_1-\rho({\mathcal{W}}_{12})x_2
    \end{pmatrix}\\  
&=& \begin{pmatrix}
    \widehat{P}^{\dag}\widehat{R}x_1-\dfrac{1}{\rho({\mathcal{W}}_{12})}\widehat{P}^{\dag}\widehat{S}x_1-\widehat{\mathcal{P}}^{\dag}\widehat{\mathcal{R}}x_1+\dfrac{1}{\rho({\mathcal{W}}_{12})}\widehat{\mathcal{P}}^{\dag}\widehat{\mathcal{S}}x_1\\
    0
    \end{pmatrix}\\
    & \geq & \begin{pmatrix}
    \dfrac{1}{\rho({\mathcal{W}}_{12})}(\widehat{P}^{\dag}\widehat{R}-\widehat{\mathcal{P}}^{\dag}\widehat{\mathcal{R}})x_1-\dfrac{1}{\rho({\mathcal{W}}_{12})}(\widehat{P}^{\dag}\widehat{S}-\widehat{\mathcal{P}}^{\dag}\widehat{\mathcal{S}})x_1\\ 
    0
    \end{pmatrix}\\
    &=& \begin{pmatrix}
   \dfrac{1}{\rho({\mathcal{W}}_{12})}\left(\widehat{P}^{\dag}(\widehat{P}-\widehat{A})-\widehat{\mathcal{P}}^{\dag}(\widehat{\mathcal{P}}-\widehat{\mathcal{A}})\right)x_1 \\ 
    0
    \end{pmatrix}\\
    &=& \begin{pmatrix}
    \dfrac{1}{\rho({\mathcal{W}}_{12})}(\widehat{\mathcal{P}}^{\dag}\widehat{\mathcal{A}}-\widehat{P}^{\dag}\widehat{A})x_1 \\ 
    0
    \end{pmatrix}\geq 0.
\end{eqnarray*}
Hence $\rho({\mathcal{W}}_{12}){\bf x}\leq {\bf W_{12}}{\bf x}$. By Theorem 
\ref{2.1.2}, we therefore have $\rho({\mathcal{W}}_{12})\leq\rho({W}_{12})<1.$
\end{proof}

\begin{remark}\label{remGTHTcomp2}
If the conditions are again replaced by $\widehat{P}^{\dag}\widehat{R}\geq \widehat{\mathcal{P}}^{\dag}\widehat{\mathcal{R}}$ and $\widehat{P}^{\dag}\widehat{A}\geq \widehat{\mathcal{P}}^{\dag}\widehat{\mathcal{A}}$, then the HT-ADS scheme performs better than the GT-ADS scheme.
\end{remark}

As a consequence of Theorem \ref{wwcomp2}, we have the the following corollary. 
\begin{corollary}
Let $A = P_1-R_1+S_1$ be a double weak regular  splitting
and $A = P_2-R_2+S_2$ be a double regular splitting of a monotone matrix $A$. Suppose that  $1 \notin \sigma(S_2P_1^{-1})$, $-1 \notin \sigma(R_2P_1^{-1})$, $\widehat{A}^{-1}\geq 0$ and $\widehat{\mathcal{A}}^{-1}\geq 0$. If $\widehat{\mathcal{P}}^{-1}\widehat{\mathcal{R}}\geq \widehat{P}^{-1}\widehat{R}$ and $ \widehat{\mathcal{P}}^{-1}\widehat{\mathcal{A}}\geq \widehat{P}^{-1}\widehat{A}$, then $\rho({\mathcal{W}}_{12})\leq\rho({W}_{12})<1.$
\end{corollary}

\section{Numerical results}

In this section, 
numerical results are given to demonstrate the accuracy and effectiveness of the proposed ADS schemes.  The computations are carried out using Mathematica 10.0 and MATLAB R2018a on an intel(R) Core(TM)i5, 2.5GHz, 16GB RAM. The stopping criteria is  $\|x_k-x_{k-1}\|\leq \epsilon = 10^{-7}$. We have considered two different examples: one  for the case of non-singular matrices and the other for rectangular matrices.

\begin{example}[Example $4.1$, \cite{msx}]\label{pde} Applying second order five-point central difference scheme for the following two-dimensional convection-diffusion equation:
\begin{equation*}
    -\frac{\partial^2 u}{\partial^2 x}-\frac{\partial^2 u}{\partial^2 y}+\frac{\partial u}{\partial x}+2\frac{\partial u}{\partial y}=\sin x,~(x,y)\in \Omega=[0,1]\times [0,1],
\end{equation*}
 we obtain a system of linear equations $Ax=b$, where $A$ is non-singular matrix. The discretization is made using uniform grids with $N_x\times N_y$ interior nodes, where the solution is known at the boundary. Therefore, the coefficient matrix $A$ is of the form $$
 A=I_y\kronecker J_x +J_y\kronecker I_x.$$
Here $\kronecker$ is the  Kronecker product, and the matrices $J_x$ and $J_y$ are tridiagonal matrices of order $N_x$ and $N_y$ respectively, i.e.,
\begin{equation*}
    J_x = tridiagonal\left(-2-h_x, 8 ,h_x-2\right) \mbox{ and } J_y = tridiagonal\left(-2-2h_y, 0 ,2h_y-2 \right),
\end{equation*}
where $h_x$ and $h_y$ are the uniform step size along $x$ and $y$ directions, respectively. Similarly, the identity matrices $I_x$ and $I_y$ are of the dimension $N_x$ and $N_y$, respectively. We can observe $A$ is not a symmetric matrix but its diagonally dominant block tridiagonal matrix hence irreducible. This properties of matrices implies they are monotonic, which is very useful while investigating our theoretical findings by numerical experiments.
The proposed TG-ADS scheme is compared with the iterative methods  of \cite{linw}, \cite{maxi}, \cite{shn},  \cite{zhwy} and \cite{zhwe}. The  Table \ref{tab:tablepde} compares  the residual norm ($\|r_k\|=\|b-Ax_{k}\|$),   error norm ($\|e_k\|=\|A^{\dag}b-x_k\|$) and Mean Time(MT).  The symbol $(-)$ represents that the TG-ADS scheme does not converge within the maximum allowed iteration (4000). 
Figure \ref{fig:fig1} presents the computational time of the present ADS scheme which outperforms the iteration schemes used in Table 1. The same figure shows that the computational time for the increasing  size of the discretization matrices (by reducing the step length $h$).  The computational time of the TG-ADS scheme is consistently lesser than the existing schemes for all size of matrices. 
\end{example}
\begin{figure}[H]
	\centering
	\includegraphics[height=3.2in,width=6.4in]{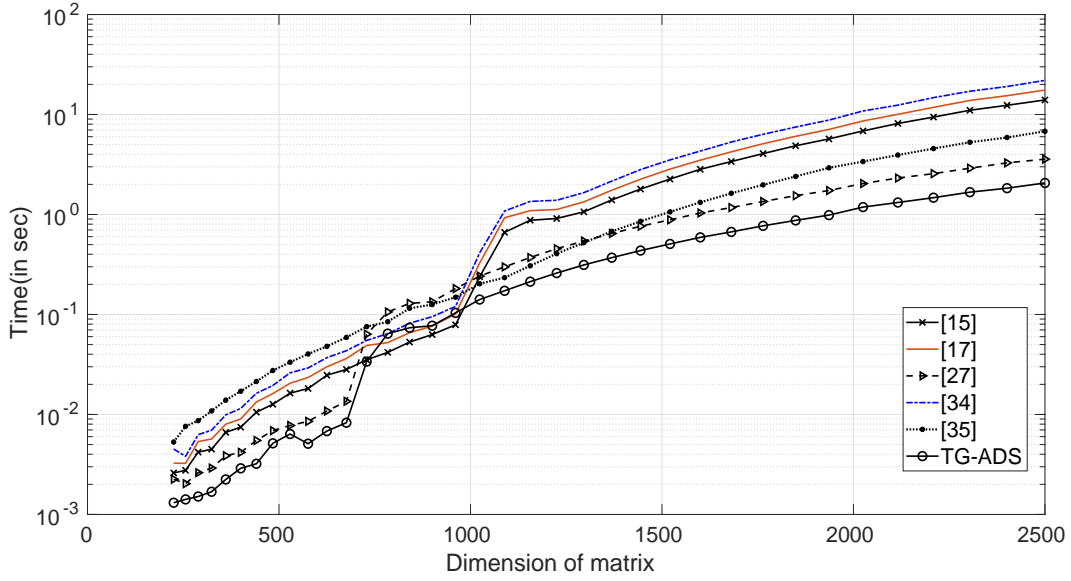}
		\vspace{-0.4cm}
\caption{Comparison of existing methods with the TG-ADS scheme }
	\label{fig:fig1}
\end{figure}

\begin{table}[H]
    \centering
     \caption{Comparison analysis of different schemes for $\epsilon=10^{-7}$}
    \begin{tabular}{cccccc}
    \hline
     Order of $A$ &  Method & n & $\|r_k\|_{2}$ & $\|e_k\|_{2}$ & MT\\
       \hline
       \multirow{6}{*} {$15 \times 15$} & Method of \cite{linw} & $453$ & $4.6057e^{-8}$ &$2.3197e^{-6}$ & $0.00955$\\
       & Method of \cite{maxi} & $568$ & $5.9575e^{-8}$ &$2.9791e^{-6}$ & $0.01525$\\
        & Method of \cite{shn} & $161$ & $2.6862e^{-8}$ &$6.4999e^{-7}$ & $0.00798$\\
         & Method of \cite{zhwy} & $701$ & $7.6163e^{-8}$ &$3.7715e^{-6}$ & $0.01704$\\
       
                           & TG-ADS  & $93$ & $1.1769e^{-8}$ & $2.8399e^{-7}$ & $0.00544$\\
           \hline
           \multirow{6}{*} {$25\times 25$} & Method of \cite{linw} & $1230$ & $4.7809e^{-8}$ &$6.3031e^{-6}$ & $0.16492$\\
           & Method of \cite{maxi} & $1514$ & $6.1655e^{-8}$ &$8.0690e^{-6}$ & $0.18487$\\
                   & Method of \cite{shn} & $283$ & $1.7583e^{-8}$ &$1.0791e^{-6}$ & $0.11376$\\
                    & Method of \cite{zhwy} & $1904$ & $7.7942e^{-8}$ &$1.0127e^{-5}$ & $0.25060$\\
           & TG-ADS & $162$ & $8.9456e^{-9}$ & $5.4889e^{-7}$ & $0.07400$\\
            \hline
           \multirow{6}{*} { $35\times 35$} & Method of \cite{linw} & $2399$ & $4.8831e^{-8}$ &$1.2277e^{-5}$ & $1.94247$\\
           & Method of \cite{maxi} & $3007$ & $6.2089e^{-8}$ & $1.5513e^{-5}$ & $2.39230$\\
            & Method of \cite{shn} & $408$ & $1.3870e^{-8}$ & $1.6075e^{-6}$ & $0.62750$\\
       & Method of \cite{zhwy} & $3716$ & $7.8517e^{-8}$ & $1.9509e^{-5}$ & $2.97910$\\
                       & TG-ADS  & $233$ & $7.2425e^{-9}$ & $8.4203e^{-7}$ & $0.37433$\\
            \hline
           \multirow{6}{*} {$50\times 50$} & Method of \cite{linw} & $-$ & $-$ &$-$ & $-$\\
           & Method of \cite{maxi} & $-$ & $-$ & $-$ & $-$\\
            & Method of \cite{shn} & $602$ & $9.8102e^{-9}$ & $2.2771e^{-6}$ & $3.92001$\\
       & Method of \cite{zhwy} & $-$ & $-$ &$-$ & $-$\\
                        & TG-ADS  & $343$ & $5.3138e^{-9}$ & $1.2329e^{-6}$ & $2.18799$\\
           \hline
    \end{tabular}
    \label{tab:tablepde}
\end{table}

Next, we will perform a few computational experiments to understand
the efficiency of the preconditioners induced by the ADS schemes.
The preconditioning matrix which will modify the original matrix such that the new matrix will be closer to the identity matrix or at least that the eigenvalues of the new matrix are clustered together, see \cite{wathen15} by Wathen in 2015. Hence, we can compute $\|I - LA\|$ with respect to different preconditioning matrix $L$ and compare with the $\|I - A\|$ to identify the efficient preconditioning matrix. In Table  \ref{tab:tableprecondnew11} and \ref{tab:tableprecond}, we have compared the efficiency of the preconditioners along with that we have observed the decrease in condition number of the coefficient matrix with respect to the increase in efficiency of the preconditioners induced by the ADS schemes.



\begin{table}[H]
    \centering
     \caption{Comparison of preconditioners}
    \begin{tabular}{cccccc}
    \hline
     Order & System & n & Time & Condition number & Efficiency\\
       \hline
       \multirow{3}{*} {$15\times 15$} & ($A, b$) & $178$ &$0.22885$ & $100.3994$  & $0.9808$\\
       & $(\widehat{A}, \widehat{b})$ & $106$ & $0.16936$ & $91.5894$  & $0.9791$ \\
       & $(\widehat{\mathcal{ A}}, \widehat{\mathpzc{b}})$ & $93$ & $0.13691$& $53.0487$  & $0.9657$\\
        \hline
        \multirow{3}{*} {$25\times 25$} &  $(A, b)$ & $312$ & $5.34683$ & $266.0636$  & $0.9927$ \\
         & $(\widehat{A}, \widehat{b})$ & $186$ & $3.98814$ & $ 242.2502$  & $0.9920$ \\
         & $(\widehat{\mathcal{ A}}, \widehat{\mathpzc{b}})$ & $163$ & $3.48179$& $140.4851$ & $0.9816$ \\
        \hline
       \multirow{3}{*} {$35\times 35$} &  $(A, b)$ & $451$  & $43.16318$ & $510.6155$ & $0.9962$ \\
        & $(\widehat{A}, \widehat{b})$  & $268$ & $37.29997$ & $464.6237$ & $0.9958$ \\
        & $(\widehat{\mathcal{ A}}, \widehat{\mathpzc{b}})$ & $235$ & $32.84324$& $269.9276$ & $0.9873$ \\
        \hline
 \multirow{3}{*} {$50\times 50$} &  $(A, b)$ & $665$  & $412.50718$ & $1025.400$  & $0.9981$\\
     & $(\widehat{A}, \widehat{b})$ &   $395$ & $359.45053$ & $932.6497$ & $0.9979$ \\
     & $(\widehat{\mathcal{ A}}, \widehat{\mathpzc{b}})$ & $346$ & $336.15932$ & $542.8390$ & $0.9897$ \\
        \hline
    \end{tabular}
    \label{tab:tableprecondnew11}
\end{table}

For the computations in Table  \ref{tab:tableprecondnew11}, we have selected a second splitting $A = P_2 -R_2 + S_2$ such that that HT-ADS scheme converges faster than TG-ADS scheme. As a result, it shows that the preconditioned system \eqref{pc2} is better than the earlier one \eqref{pc1}. The comparison theorem (i.e., Theorem~\ref{wwcomp2}) served the sufficient conditions under which the faster convergence of the HT-ADS scheme is guaranteed. In particular, one can observe that the condition number of $A$ reduces from 1025.400 to 542.839 when matrix size is 2500 for the  preconditioned system \eqref{pc2} induced by the  HT-ADS scheme. The purpose of the last column of the table is crucial in order to measure the efficiency of the preconditioning matrix by computing the norm of the difference of the matrix or the preconditioned matrices from the identity matrix. The minimum norm will assure that the preconditioned matrix is the closest to identity matrix and confirm the corresponding preconditioner is the most efficient and its resulting system have the least condition number. For all sizes of matrices, considered in the table, the HT-ADS scheme preconditioner is consistently efficient and the condition number is less. Due to this effect, HT-ADS scheme converges with the least number of iterations and computational time.\\
 
 In  Table \ref{tab:tableprecond}, we have a  second splitting $A = P_2 -R_2 + S_2$ (in the complementary class of the splitting considered in Table \ref{tab:tableprecondnew11}) such that the TG-ADS scheme converges faster than the HT-ADS scheme. On the contrary to the results in Table 2, the TG-ADS scheme induces the efficient preconditioner and the preconditioned linear system has the least condition number. For this case, the guaranteed conditions on the splittings are reported in Remark \ref{remGTHTcomp2}, which are the sufficient conditions.
 Simultaneously, the iteration numbers and computational times are the least for the most efficient preconditioner, which has been consistently observed for the matrices of sizes 225, 625, 1225 and 2500 derived form the discretized PDE.
 
\begin{table}[H]\label{table3}
    \centering
     \caption{Comparison of preconditioners: complementary case of Table 2} 
    \begin{tabular}{cccccc}
    \hline
     Order & Systems & n & Time & Condition number  & Efficiency \\
       \hline
       \multirow{3}{*} {$15\times 15$} & $ (A, b) $ & $178$  &$0.21937$ & $100.3994$ & $0.9808$ \\
       & $(\widehat{A}, \widehat{b})$ & $96$  & $0.18625$ & $60.3614$ & $0.9696$ \\
       & $(\widehat{\mathcal{A}}, \widehat{\mathpzc{b}})$ & $101$ & $0.15694$& $75.7562$ & $0.9752$ \\
        \hline
        \multirow{3}{*} {$25\times 25$} & $(A, b)$ & $312$ & $4.79311$ & $266.0636$ & $0.9927$ \\
         & $(\widehat{A}, \widehat{b})$ & $168$ & $3.68142$ & $159.2569$ & $0.9882$  \\
         & $(\widehat{\mathcal{ A}}, \widehat{\mathpzc{b}})$ & $177$ & $3.87634$& $199.9162$ & $0.9905$ \\
        \hline
       \multirow{3}{*} {$35\times 35$} & $(A, b)$ & $451$  & $39.98544$ & $510.6155$  & $0.9962$ \\
        & $(\widehat{A}, \widehat{b})$ & $242$ & $32.91968$ & $305.4747$  & $0.9938$ \\
        & $(\widehat{\mathcal{ A}}, \widehat{\mathpzc{b}})$ & $255$ & $32.91633$& $383.2356$  & $0.9950$ \\
        \hline
 \multirow{3}{*} {$50\times 50$} & $ (A, b) $ & $665$  & $433.69228$ & $1025.400$  & $0.9981$ \\
     & $(\widehat{A}, \widehat{b})$ &   $356$ & $333.72665$ & $613.5452$  & $0.9969$  \\
     & $(\widehat{\mathcal{ A}}, \widehat{\mathpzc{b}})$ & $375$ & $360.78457$ & $769.1473$  & $0.9975$ \\
        \hline
    \end{tabular}
    \label{tab:tableprecond}
\end{table}

The following example demonstrates Theorem \ref{altcomp}, and  also used to generate large rectangular matrices that are used for the computation in Table 4. 

\begin{example}\label{ex3.7}
Let  $ A = \begin{bmatrix}
  1 & 18 & -\frac{1}{2} & -\frac{1}{4} & -\frac{1}{8} & -\frac{1}{16} & -\frac{1}{32} & -\frac{1}{64} & -\frac{1}{128} & 0
   \\
 0 & -\frac{1}{2} & 14 & -\frac{1}{2} & -\frac{1}{4} & -\frac{1}{8} & -\frac{1}{16} & -\frac{1}{32} & -\frac{1}{64} & 0 \\
 0 & -\frac{1}{4} & -\frac{1}{2} & 20 & -\frac{1}{2} & -\frac{1}{4} & -\frac{1}{8} & -\frac{1}{16} & -\frac{1}{32} & 0 \\
 0 & -\frac{1}{8} & -\frac{1}{4} & -\frac{1}{2} & 11 & -\frac{1}{2} & -\frac{1}{4} & -\frac{1}{8} & -\frac{1}{16} & 0 \\
 0 & -\frac{1}{16} & -\frac{1}{8} & -\frac{1}{4} & -\frac{1}{2} & 14 & -\frac{1}{2} & -\frac{1}{4} & -\frac{1}{8} & 0 \\
 0 & -\frac{1}{32} & -\frac{1}{16} & -\frac{1}{8} & -\frac{1}{4} & -\frac{1}{2} & 19 & -\frac{1}{2} & -\frac{1}{4} & 0 \\
 0 & -\frac{1}{64} & -\frac{1}{32} & -\frac{1}{16} & -\frac{1}{8} & -\frac{1}{4} & -\frac{1}{2} & 19 & -\frac{1}{2} & 0 \\
 0 & -\frac{1}{128} & -\frac{1}{64} & -\frac{1}{32} & -\frac{1}{16} & -\frac{1}{8} & -\frac{1}{4} & -\frac{1}{2} & 19 & 1
   \\
\end{bmatrix}$\\[1ex]
$= P_1-R_1+S_1$ be a double proper weak regular  splitting, where
\begin{align*}
   P_1 =   \begin{bmatrix}
\frac{871705488637}{33342018661} & 470 & -\frac{1}{2} & -\frac{1}{4} & -\frac{1}{8} & -\frac{1}{16} & -\frac{1}{32} &
   -\frac{1}{64} & -\frac{1}{128} & \frac{72320000}{4763145523} \\
 \frac{31485697615}{33342018661} & -\frac{1}{2} & 479 & -\frac{1}{2} & -\frac{1}{4} & -\frac{1}{8} & -\frac{1}{16} &
   -\frac{1}{32} & -\frac{1}{64} & \frac{181040000}{4763145523} \\
 \frac{31551183050}{100026055983} & -\frac{1}{4} & -\frac{1}{2} & 430 & -\frac{1}{2} & -\frac{1}{4} & -\frac{1}{8} &
   -\frac{1}{16} & -\frac{1}{32} & \frac{642880000}{14289436569} \\
 \frac{3176280160}{14289436569} & -\frac{1}{8} & -\frac{1}{4} & -\frac{1}{2} & 315 & -\frac{1}{2} & -\frac{1}{4} &
   -\frac{1}{8} & -\frac{1}{16} & \frac{1582016000}{14289436569} \\
 \frac{1851285950}{14289436569} & -\frac{1}{16} & -\frac{1}{8} & -\frac{1}{4} & -\frac{1}{2} & 429 & -\frac{1}{2} &
   -\frac{1}{4} & -\frac{1}{8} & \frac{3192346000}{14289436569} \\
 \frac{180765200}{4763145523} & -\frac{1}{32} & -\frac{1}{16} & -\frac{1}{8} & -\frac{1}{4} & -\frac{1}{2} & 327 &
   -\frac{1}{2} & -\frac{1}{4} & \frac{1100989120}{4763145523} \\
 \frac{152460000}{4763145523} & -\frac{1}{64} & -\frac{1}{32} & -\frac{1}{16} & -\frac{1}{8} & -\frac{1}{4} & -\frac{1}{2}
   & 514 & -\frac{1}{2} & \frac{3338973990}{4763145523} \\
 \frac{68320000}{4763145523} & -\frac{1}{128} & -\frac{1}{64} & -\frac{1}{32} & -\frac{1}{16} & -\frac{1}{8} &
   -\frac{1}{4} & -\frac{1}{2} & 446 & \frac{111914626295}{4763145523} \\
\end{bmatrix}
\end{align*}
and 

\begin{align*}
   R_1 =  \begin{bmatrix}
 \frac{628772602482}{33342018661} & 339 & 0 & 0 & 0 & 0 & 0 & 0 & 0 & \frac{54240000}{4763145523} \\
 \frac{94457092845}{133368074644} & 0 & \frac{1395}{4} & 0 & 0 & 0 & 0 & 0 & 0 & \frac{135780000}{4763145523} \\
 \frac{15775591525}{66684037322} & 0 & 0 & \frac{615}{2} & 0 & 0 & 0 & 0 & 0 & \frac{160720000}{4763145523} \\
 \frac{794070040}{4763145523} & 0 & 0 & 0 & 228 & 0 & 0 & 0 & 0 & \frac{395504000}{4763145523} \\
 \frac{925642975}{9526291046} & 0 & 0 & 0 & 0 & \frac{1245}{4} & 0 & 0 & 0 & \frac{798086500}{4763145523} \\
 \frac{135573900}{4763145523} & 0 & 0 & 0 & 0 & 0 & 231 & 0 & 0 & \frac{825741840}{4763145523} \\
 \frac{114345000}{4763145523} & 0 & 0 & 0 & 0 & 0 & 0 & \frac{1485}{4} & 0 & \frac{5008460985}{9526291046} \\
 \frac{51240000}{4763145523} & 0 & 0 & 0 & 0 & 0 & 0 & 0 & \frac{1281}{4} & \frac{80363610579}{4763145523} \\
\end{bmatrix}.
\end{align*}

Again, $A = P_2-R_2+S_2$ is a double proper regular  splitting of $A$ which satisfies  $N(S_2)\supseteq N(P_2)$, $R(S_2)\subseteq R(P_2)$, $||S_2P_1^{\dag}||<1$ and $\widehat{A}^{\dag}\geq 0$. Here, we have
\begin{align*}
    P_2 = \begin{bmatrix}
\frac{775256593861}{33342018661} & 418 & -\frac{1}{2} & -\frac{1}{4} & -\frac{1}{8} & -\frac{1}{16} & -\frac{1}{32} &
   -\frac{1}{64} & -\frac{1}{128} & \frac{64000000}{4763145523} \\
 \frac{81253413200}{100026055983} & -\frac{1}{2} & 414 & -\frac{1}{2} & -\frac{1}{4} & -\frac{1}{8} & -\frac{1}{16} &
   -\frac{1}{32} & -\frac{1}{64} & \frac{467200000}{14289436569} \\
 \frac{30781642000}{100026055983} & -\frac{1}{4} & -\frac{1}{2} & 420 & -\frac{1}{2} & -\frac{1}{4} & -\frac{1}{8} &
   -\frac{1}{16} & -\frac{1}{32} & \frac{627200000}{14289436569} \\
 \frac{4179316000}{14289436569} & -\frac{1}{8} & -\frac{1}{4} & -\frac{1}{2} & 411 & -\frac{1}{2} & -\frac{1}{4} &
   -\frac{1}{8} & -\frac{1}{16} & \frac{2081600000}{14289436569} \\
 \frac{1784372000}{14289436569} & -\frac{1}{16} & -\frac{1}{8} & -\frac{1}{4} & -\frac{1}{2} & 414 & -\frac{1}{2} &
   -\frac{1}{4} & -\frac{1}{8} & \frac{3076960000}{14289436569} \\
 \frac{234760000}{4763145523} & -\frac{1}{32} & -\frac{1}{16} & -\frac{1}{8} & -\frac{1}{4} & -\frac{1}{2} & 419 &
   -\frac{1}{2} & -\frac{1}{4} & \frac{1429856000}{4763145523} \\
 \frac{123200000}{4763145523} & -\frac{1}{64} & -\frac{1}{32} & -\frac{1}{16} & -\frac{1}{8} & -\frac{1}{4} & -\frac{1}{2}
   & 419 & -\frac{1}{2} & \frac{2698160800}{4763145523} \\
 \frac{64000000}{4763145523} & -\frac{1}{128} & -\frac{1}{64} & -\frac{1}{32} & -\frac{1}{16} & -\frac{1}{8} &
   -\frac{1}{4} & -\frac{1}{2} & 419 & \frac{105139239923}{4763145523} \\
\end{bmatrix}
\end{align*}

and 

\begin{align*}
    R_2 = \begin{bmatrix}
\frac{695544914250}{33342018661} & 375 & 0 & 0 & 0 & 0 & 0 & 0 & 0 & \frac{60000000}{4763145523} \\
 \frac{25391691625}{33342018661} & 0 & 375 & 0 & 0 & 0 & 0 & 0 & 0 & \frac{146000000}{4763145523} \\
 \frac{9619263125}{33342018661} & 0 & 0 & 375 & 0 & 0 & 0 & 0 & 0 & \frac{196000000}{4763145523} \\
 \frac{1306036250}{4763145523} & 0 & 0 & 0 & 375 & 0 & 0 & 0 & 0 & \frac{650500000}{4763145523} \\
 \frac{557616250}{4763145523} & 0 & 0 & 0 & 0 & 375 & 0 & 0 & 0 & \frac{961550000}{4763145523} \\
 \frac{220087500}{4763145523} & 0 & 0 & 0 & 0 & 0 & 375 & 0 & 0 & \frac{1340490000}{4763145523} \\
 \frac{115500000}{4763145523} & 0 & 0 & 0 & 0 & 0 & 0 & 375 & 0 & \frac{2529525750}{4763145523} \\
 \frac{60000000}{4763145523} & 0 & 0 & 0 & 0 & 0 & 0 & 0 & 375 & \frac{94102588500}{4763145523} \\
\end{bmatrix}.
\end{align*}
Therefore, $0.9720 = \rho(W_{12})\leq 0.9752 = \rho(T_2)<1.$  The computational performance of the TG-ADS scheme with the  double iteration scheme \eqref{eq2} is summarized in Table \ref{tab:table1}.
\end{example}

\begin{table}[H]\label{hybds}
    \centering
     \caption{Comparison analysis for rectangular matrices}
    \begin{tabular}{ccccccc}
    \hline
     Order & Method & n & $\|r_n\|$  & $\|e_n\|$ & $\rho$ & MT\\
       \hline
       \multirow{2}{*} {$8\times 10$} & TG-ADS & $626$ & $1.0675e^{-6}$ & $9.8514e^{-8}$ & $0.9720$ & $0.00469$ \\
       & Method of \cite{jmp} & $707$ & $1.0592e^{-6}$ & $9.7765e^{-8}$ & $0.9752$ & $0.00625$ \\
        \hline
        \multirow{2}{*} {$18\times 20$} & TG-ADS  & $675$ & $8.9624e^{-7}$ & $9.7691e^{-8}$ & $0.9751$ & $0.01250$  \\
         & Method of \cite{jmp} & $897$ & $9.0380e^{-7}$ & $9.8409e^{-8}$ & $0.9812$ & $0.01406$  \\
        \hline
       \multirow{2}{*} {$28\times 30$} & TG-ADS  & $678$ & $9.8211e^{-7}$ & $9.9963e^{-8}$ & $0.9754$ & $0.02609$  \\
        & Method of \cite{jmp}  & $983$ & $9.8307e^{-7}$ & $9.8492e^{-8}$ & $0.9826$ & $0.03403$  \\
        \hline
 \multirow{2}{*} {$48\times 50$} & TG-ADS  & $727$ & $9.9325e^{-7}$ & $9.9956e^{-8}$ & $0.9812$ & $0.04712$  \\
     & Method of \cite{jmp} &   $1188$ & $9.9534e^{-7}$ & $9.9867e^{-8}$ & $0.9963$ & $0.06548$  \\
        \hline
    \end{tabular}
    \label{tab:table1}
\end{table}

We have selected four rectangular matrices by the column extension of the diagonally dominant matrices of sizes 8, 18, 28 and 48, respectively. Example~\ref{ex3.7} is explained for a diagonally dominant matrix of size 8, whose columns are extended to 10. Explicitly, we have obtained two double splittings, which satisfy the necessary conditions such that the preconditioned matrix induced by the TG-ADS scheme has a convergent double proper regular(weak) splitting. This rectangular matrix $A$ of size $8 \times 10$ is semi-monotone. Similarly, the rest of the three matrices can be shown as semi-monotone matrices. We have computed the error norm to make sure that the approximate solution is achieved within the required digit accuracy before the stopping criteria meet the tolerance.

We next generate a $38\times 40$ semi-monotone matrix as in Example \ref{ex3.7} to illustrate the residual and error of different iterative schemes.  Residual $\&$ error norms are plotted against the iteration number in figure \ref{fig:fig2}. 

\begin{figure}[H]
~~~~~~~(a)~~~~~~~~~~~~~~~~~~~~~~~~~~~~~~~~~~~~~~~~~~~~~~~~~~~~~~(b)
	\centering
	\includegraphics[height=2.5in,width=3.1in]{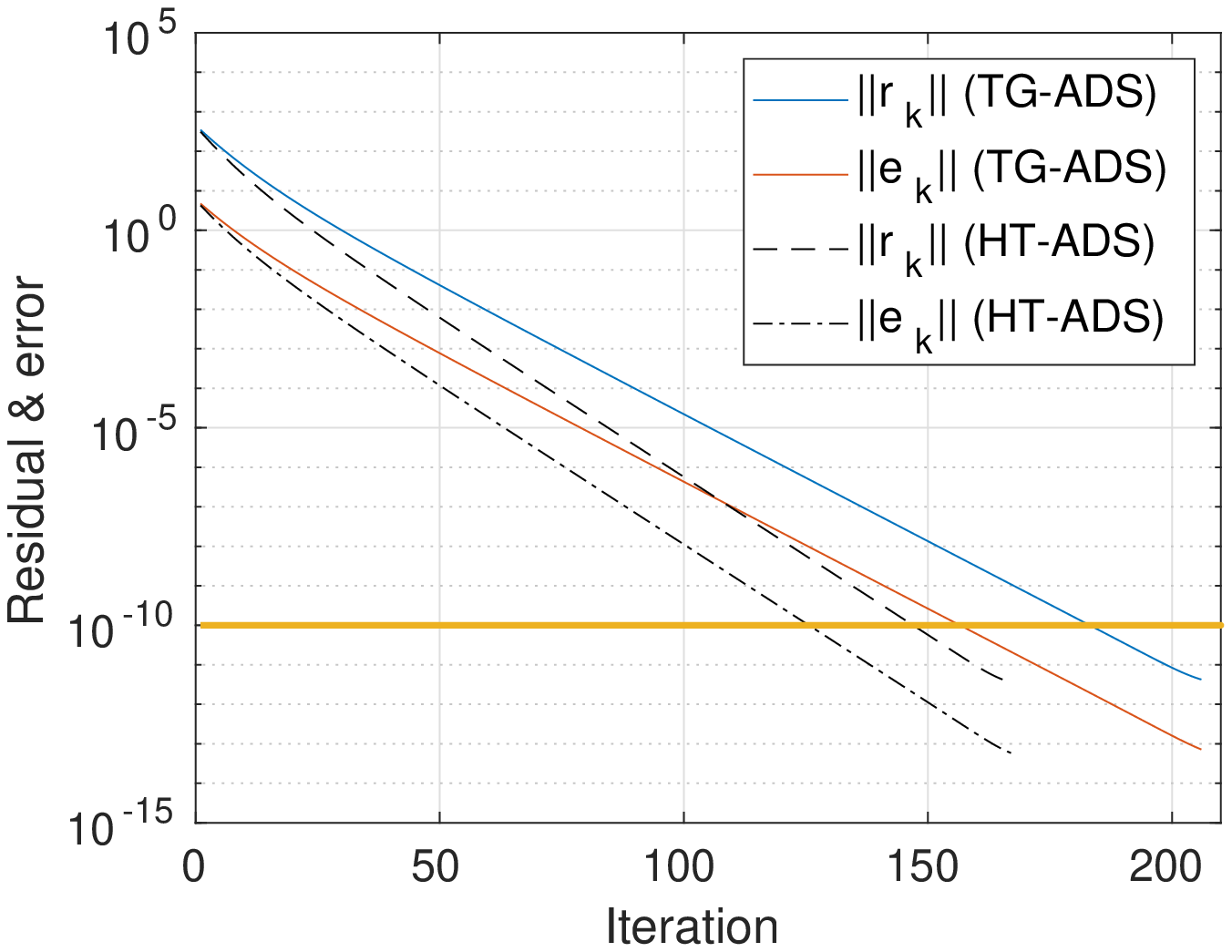}
	\includegraphics[height=2.5in,width=3.1in]{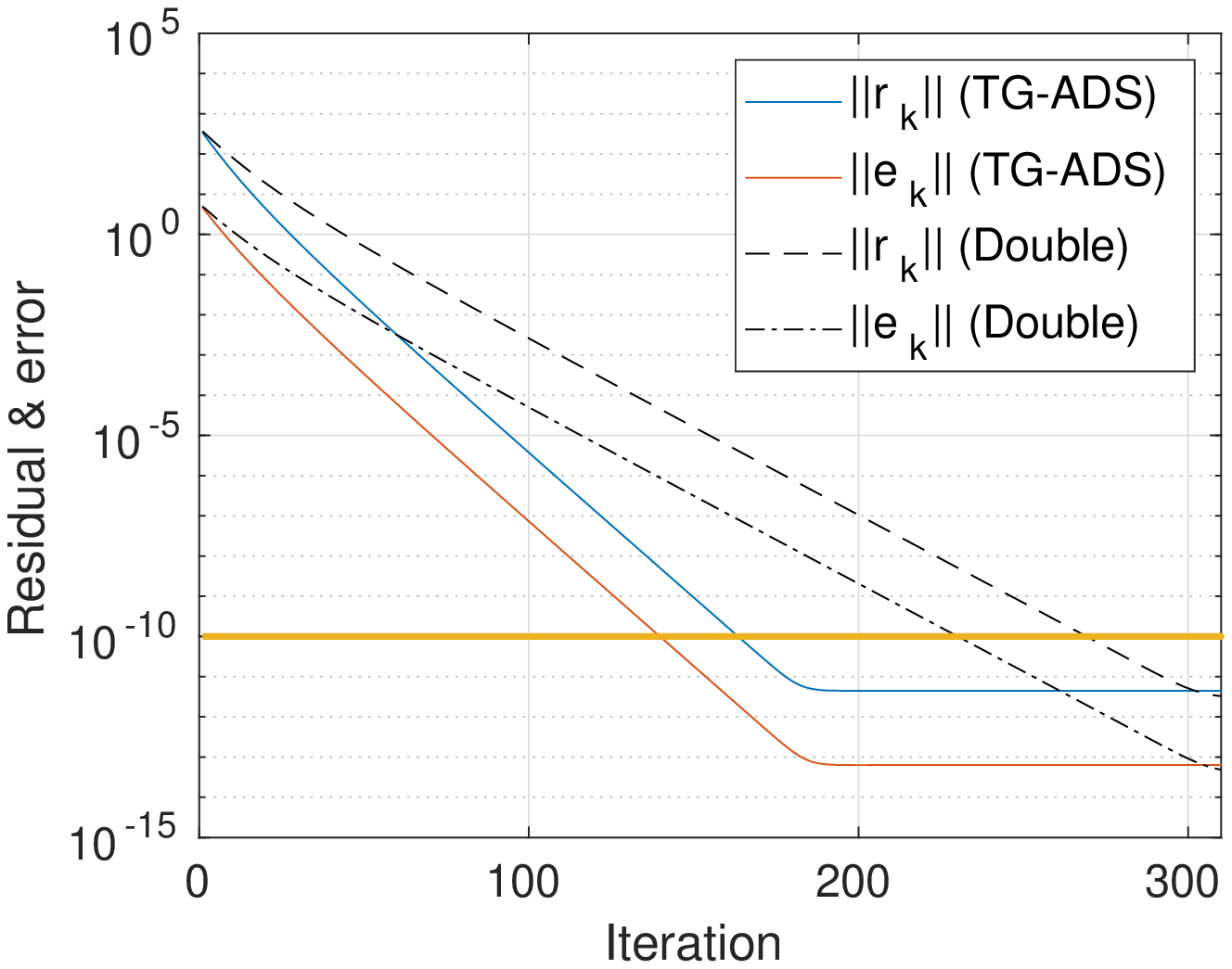}
	\vspace{-0.4cm}
    \caption{(a) Comparison of convergence in errors and residuals of TG and HT-ADS schemes.\\ (b) Comparison of Double iteration scheme \eqref{eq2} with TG-ADS scheme}
    \label{fig:fig2}
\end{figure}

\section{Conclusions}
In this paper, we have proposed a new alternating scheme using double splittings (HT-ADS scheme) like the one introduced by  Li {\it et al.} \cite{li} in 2019, and studied the extension
of both the schemes to rectangular matrix setting. The important findings are summarized as
follows:

\begin{itemize}

\item Formulation of the proposed schemes: TG-ADS scheme and HT-ADS scheme, are shown in Section \ref{sub3.01} in the rectangular matrix setting.  Then, the convergence analysis is carried out in \ref{sec:cgs} for the class of double proper weak regular splittings. This is done by considering another preconditioned linear system which is induced by the ADS scheme. 

\item The significance of introducing  ADS schemes are studied next. In this context, we have established several analytical results which justifies the importance by showing the faster convergence of the ADS schemes. We have also presented a few results which will guide us to choose a particular ADS scheme in case we have more than one same type of ADS schemes. More importantly, we have shown that one ADS scheme outperforms the other for a certain case. This is proved in Theorem \ref{wwcomp2}.

\item As illustrated in Example \ref{pde}, there are substantial examples of linear systems of PDEs. Our numerical experiments with test matrices from different applications suggest that the ADS schemes are fairly robust. These computations also show that the ADS scheme performs better than some other existing schemes in the literature. Residual and error norms of the ADS scheme are monotonically convergent and faster than the double iteration scheme.

\end{itemize}

\section{Acknowledgements}

The authors thank Vaibhav Shekhar, National Institute of Technology Raipur for his valuable comments and suggestions.
The last author acknowledges the support provided
by Science and Engineering Research Board, Department of Science and Technology, New Delhi, India, under the grant number
MTR/2017/000174.



\begin{thebibliography}{99}

\bibitem{balim}
Baliarsingh, A. K., Mishra, D., Comparison results for proper nonnegative splittings of matrices, Results. Math. 71 (2017) 93--109.



 \bibitem{bai:computing}
 Bai, Z.-Z., Benzi, M., Chen, F., Modified HSS iteration methods for a class of complex symmetric linear systems, Computing 87 (2010) 93--111.

\bibitem{benzi:2009} Benzi, M.,  A generalization of the Hermitian and skew-Hermitian splitting iteration, SIAM J. Matrix Anal. Appl.  31(2) (2009) 360--374.


\bibitem{benz} Benzi, M., Szyld, D. B., Existence and uniqueness of splittings for stationary iterative methods with applications to alternating methods, Numer. Math. 76(3) (1997) 309--321.


\bibitem{berp}
Berman, A., Plemmons, R. J., Cones and iterative methods for best least squares solutions of linear systems, SIAM J. Numer. Anal. 11(1) (1974) 145--154.


\bibitem{bern}
Berman, A., Neumann, M.,
Proper splittings of rectangular matrices,
SIAM J. Appl. Math. 31(2) (1976)  307--312.

\bibitem{bpn}
Berman, A., Plemmons, R. J., Nonnegative Matrices in the Mathematical Sciences, SIAM, Philadelphia,
(1994).


\bibitem{BIR-VER:1959}
Birkhoff, G., Varga, R. S., Implicit alternating direction methods, Trans. Amer. Math. Soc. 92(1959) 13--24.
 
\bibitem{boyd:2010}
 Boyd, S., Parikh, N., Chu, E., Peleato, B., Eckstein, J., Distributed optimization and statistical learning via the alternating direction method of multipliers, Found. Trends Mach. Learn. 3(1)    (2010) 1--122.
 
 \bibitem{bruch-sloss:1985}
 Bruch, Jr., J. C.,  Sloss J. M., Alternating iteration and elliptic variational inequalities, Numer. Math. 47 (1985) 459--481.


\bibitem{cli1} {Climent, J.-J., Devesa, A., Perea, C.}, \emph{ Convergence results for proper splittings}, Recent Advances in Applied and Theoretical Mathematics, World
Scientific and Engineering Society Press, Singapore (2000) 39--44.

\bibitem{cli2} {Climent, J.-J., Perea, C.}, \emph{Iterative methods for least-square problems based on proper
splittings}, J. Comput. Appl. Math. 158 (2003) 43--48.


\bibitem{coltz}
Collatz, L., Functional Analysis and Numerical Mathematics, Academic Press, New York-London, 1966.

\bibitem{damm:NLAA2000}
Damm, T., Direct methods and ADI-preconditioned Krylov subspace methods for generalized Lyapunov equation, Numer. Linear Algebra Appl. 15 (9) (2008) 853--871.

\bibitem{mishalt2}
Giri, C. K., Mishra, D., Additional results on convergence of alternating iterations involving rectangular matrices, Numer. Funct. Anal. Optim. 38(2) (2017) 160--180.

\bibitem{golub}
Golub, G. H., Van Loan, C. F., Matrix Computations, The John Hopkins University Press, (1996).


\bibitem{jmp}
Jena, L., Mishra, D., Pani, S., Convergence and comparison theorems for single and double decompositions of rectangular matrices, Calcolo 51(1) (2014) 141--149.


\bibitem{kur}
Appi Reddy, K., Kurmayya, T., Comparison results for proper double splittings of rectangular matrices,  Filomat 32(6) (2018) 2273-2281.


\bibitem{lcw}
Li, C.-X., Cui, Q.-F., Wu, S.-L., Comparison theorems for single and double splittings of matrices, J. Appl. Math. Volume 2013, Article ID 827826, (2013) 4 pages. https://doi.org/10.1155/2013/827826.


\bibitem{lish}
Li, C.-X., Li, S.-H., Comparison theorems of spectral radius for splittings of matrices, J. Appl. Math. Volume 2014, Article ID 573024, (2014) 5 pages.
http://dx.doi.org/10.1155/2014/573024.


\bibitem{lws}
Li, C.-X., Wu, S.-L., Some new comparison theorems for double splittings of matrices, Appl. Math. Inf. Sci. 8(5) (2014) 2523--2526.


\bibitem{linw}
 Lin, L., Wei, Y., Zhang, N., Convergence and quotient convergence of iterative methods for solving singular linear equations with index one, Linear Algebra  Appl. 430 (2009) 1665--1674.

\bibitem{li}
Li, R., Fan, H. T., Zheng, B., An effective stationary iterative method via double splittings of matrices, Comput. Math. Appl. 77(4) (2019) 981--990.


\bibitem{maxi}
 Ma, H., Xiao, C., Convergence of nonstationary iterative methods for solving singular linear equations with index one, Numer. Funct. Anal. Optim. 38(11) (2017)
 1507--1525.


\bibitem{msx}
Miao, S.-X.,  Comparison theorems for nonnegative double splittings of different monotone matrices, J. Inf. Comput. Sci. 9(6) (2012) 1421--1428.


\bibitem{mio}
Miao, S.-X., Zheng, B., A note on double splittings of different monotone matrices, Calcolo 46(4) (2009) 261--266.

\bibitem{miga} {Migall{\'o}n, H., Migall{\'o}n, V., Penad{\'e}s, J.}, {Alternating two-stage methods for consistent linear systems with applications
to the parallel solution of Markov chains}, Adv. Eng. Softw. 41(1) (2010) 13--21.



\bibitem{mishalt1}
Mishra, D., Further study of alternating iterations for rectangular matrices, Linear Multilinear Algebra
65(8) (2017) 1566--1580.


\bibitem{dm}
Mishra, D., Nonnegative splittings for rectangular matrices, Comput. Math. Appl. 67(1) (2014) 136--144.


\bibitem{misarx}
Mishra, D., Proper weak regular splitting and its application to convergence of alternating iterations,
Filomat 32(19) (2018) 6563--6573.


\bibitem{mismis}
Mishra, N., Mishra, D., Two-stage iterations based on composite splittings for rectangular linear systems, Comput. Math. Appl. 75(8) (2018) 2746--2756.


\bibitem{ms} {Mishra, D.,  Sivakumar, K.  C.},
\emph{On splittings of matrices and nonnegative generalized inverses}, Oper. Matrices 6 (2012) 85-95.


\bibitem{nsm}
Nandi, A. K., Sahoo, J.K.,  Mishra, D., Three-step alternating iterations for index one and non-singular matrices, Numer. Algor.  84(2) (2019) 457-483.


\bibitem{neum}
Neumann, M., 3-part splittings for singular and rectangular linear systems, J. Math. Anal. Appl. 64(2) (1978) 297--318.

\bibitem{PR-ADI:1955}
Peaceman, D. W., Rachford Jr, H. H.,
The numerical solution of parabolic and elliptic differential equations, J. Soc. Indust. Appl. Math. 3 (1955) 28--41.

\bibitem{shn}
Shen, S.-Q., Huang, T.-Z., Convergence and comparison theorems for double splittings of matrices, Comput. Math. Appl. 51(12) (2006) 1751--1760.


\bibitem{shs}
 Shen, S.-Q., Huang, T.-Z., Shao, J.-L., Convergence and comparison results for double splittings of Hermitian positive definite matrices, Calcolo 44(3) (2007) 127--135.


\bibitem{shi:2014}
Shi, W., Ling, Q., Yuan, K., Wu, G.,  Yin, W., On the linear convergence of the ADMM in decentralized consensus optimization, IEEE Trans. Signal Process.  62(7) (2014) 1750--1761.


\bibitem{sjs}
Song, J., Song, Y., Convergence for nonnegative double splittings of matrices, Calcolo 48(3) (2011) 245--260.


\bibitem{var}
Varga, R. S., Matrix Iterative Analysis, Springer-Verlag, New York, Berlin, Heidelberg, (2009).


\bibitem{wanz}
Wang, X.-Z., Convergence of $H$-double splitting for $H$-matrices, Results. Math. 66 (2014) 125--135.

\bibitem{wathen15}
Wathen, A. J., Preconditioning, Acta Numer. 24 (2015) 329--376.

\bibitem{wang:NLAA18}
Wang, Z.-Q., A note on the block alternating splitting implicit iteration
method for complex saddle-point problems, Numer. Linear Algebra Appl.  25 (2018), no. 6, e2209, 15 pp



\bibitem{woz}
Wo{\'z}nicki, Z. I., Estimation of the optimum relaxation factors in partial factorization iterative methods, SIAM J. Matrix Anal. Appl. 14(1) (1993) 59--73.


\bibitem{zc}
Zhang, C.-Y., On convergence of double splitting methods for non-Hermitian positive semidefinite linear systems, Calcolo 47(2) (2010) 103--112.


\bibitem{zhwy}
 Zhang, N., Wei, Y., On the convergence of general stationary iterative methods for range-Hermitian singular linear systems, Numer. Linear Algebra  Appl. 17(1) (2010) 139--154.
 

\bibitem{zhwe}
 Zhang, N., Wei, Y., Solving EP singular linear systems, Int. J. Comput. Math. 81(11) (2004) 1395--1405.


\end{thebibliography}
\end{document}